\documentclass[12pt]{article}
\DeclareMathAlphabet\mathbfcal{OMS}{cmsy}{b}{n}
\usepackage[dvipsnames]{xcolor}
\usepackage{enumerate}
\usepackage[utf8]{inputenc}
\usepackage{amsfonts}
\usepackage{amsthm}
\usepackage{lscape}
\usepackage{mathrsfs}
\usepackage{graphicx}
\usepackage{tikz-cd}
\usepackage{calc}
\usepackage{float}
\usepackage{multirow}
\usepackage{lipsum}
\usepackage{amssymb}
\usepackage{mathrsfs}
\usepackage{mathtools}
\usepackage{mathdots}
\usepackage{fancyhdr}
\usepackage{tikz}
\usepackage{faktor}
\usepackage{setspace}
\usetikzlibrary{decorations.markings}

\theoremstyle{plain}
\usepackage{sectsty}
\usepackage{kpfonts}
\sectionfont{\fontsize{14}{15}\selectfont}
\subsectionfont{\fontsize{12}{15}\selectfont}
\newtheorem{corollary}{Corollary}

\newtheorem{lemma}{Lemma}

\newtheorem{proposition}{Proposition}
\newtheorem{remark}{Remark}

\newcommand{\Gro}{\textbf{\Gr}}
\newcommand{\Gr}{{\mbox{Gr}}}

\newcommand{\GL}{{\mbox{GL}}}
\newcommand{\Hom}{{\mbox{Hom}}}

\newcommand{\M}{{\mbox{M}}}

\newcommand{\Aut}{{\mbox{Aut}}}
\newcommand{\hol}{{\mbox{hol}}}
\newcommand{\End}{{\mbox{End}}}

\newcommand{\T}{{\mbox{T}}}
\newcommand{\Aff}{{\mbox{Aff}}}

\newcommand{\im}{{\mbox{Im}}}

\newcommand{\Man}{\textbf{-Man}}

\newcommand{\Eig}{{\mbox{Eig}}}
\newcommand{\Span}{{\mbox{Span}}}

\newtheorem*{observation*}{Observation}
\newtheorem*{theorem*}{Theorem}
\newtheorem*{claim*}{Claim}  
\usepackage{blindtext}
\title{\textbf{Geometry of universal embedding spaces for almost complex manifolds}}
\author{Gabriella Clemente}
\date{}
\begin{document}
\maketitle

\begin{abstract}
We study the geometry of universal embedding spaces for compact almost complex manifolds of a given dimension. These spaces are complex algebraic analogues of twistor spaces that were introduced by J-P.\ Demailly and H.\ Gaussier. Their original goal was the study of a conjecture made by F.\ Bogomolov, asserting the ``transverse embeddability" of arbitrary compact complex manifolds into foliated algebraic varieties. In this work, we introduce a more general category of universal embedding spaces, and elucidate the geometric structure of the integrability locus characterizing integrable almost complex structures. Our approach can potentially be used to investigate the existence (or non-existence) of topological obstructions to integrability, and to tackle Yau's Challenge.
\end{abstract} 

\tableofcontents
\addcontentsline{toc}{section}{Introduction}
\section*{Introduction}
J-P.\ Demailly and H.\ Gaussier proved that every compact almost complex manifold $(X,J_X)$ of dimension $n$ admits an embedding $F:X \hookrightarrow Z$ that is transverse to an algebraic distribution $\mathcal{D} \subset T_Z$ (cf.\ Theorem 1.2 \cite{DG}), and that if $J_X$ is integrable, $\im(\bar{\partial}_{J_X} F)$ is contained in a subvariety $\mathcal{I}$ of the Grassmannian bundle $Gr^{\mathbb{C}}(\mathcal{D},n) \to Z$ that is the isotropic locus of the torsion tensor $\theta:\mathcal{D} \times \mathcal{D} \to {T_Z}/\mathcal{D},$ $\theta(\zeta,\eta)=[\zeta,\eta] \mod{\mathcal{D}}$ (cf.\ Theorem 1.6 \cite{DG}). J-P.\ Demailly proposed using this result to study 

\textbf{Yau's Challenge:} Determine whether there exists a compact almost complex manifold of dimension at least $3$ that cannot be given a complex structure \cite{YC}. 

The proposed strategy is most promising when the homotopy classes of almost complex structures are well understood. This is the case for the (oriented) sphere $S^6$ since any almost complex structure on $S^6$ is homotopic to the well-known octonion almost complex structure.

Let $J$ be a hypothetical integrable almost complex structure on $X$ that is homotopic to $J_X.$ A (smooth) homotopy $J_t$ such that $J_0=J_X$ and $J_1=J$ produces an isotopy of transverse to $\mathcal{D}$ embeddings $F_t:(X,J_t) \hookrightarrow Z,$ giving rise, in turn, to a homotopy $\im(\bar{\partial}_{J_t} F_t (x)):[0,1] \times X \rightarrow \Gr^{\mathbb{C}}(\mathcal{D},n),$ where $\im(\bar{\partial}_{J_0} F_0 (x)) \subset \Gr^{\mathbb{C}}(\mathcal{D},n)$ and $\im(\bar{\partial}_{J_1} F_1 (x)) \subset \mathcal{I}.$ The original proposal was to study the topology of $\mathcal{I}$ relative to $\Gr^{\mathbb{C}}(\mathcal{D},n)$ independently in hopes of detecting an obstruction to the existence of the homotopy $\im(\bar{\partial}_{J_t} F_t (x))$ \cite{CIME}. It is unclear if such an obstruction exists at all. But if it did, then in the special case of a homotopically unique $J_X,$ the conclusion would be that $X$ cannot be a complex manifold. 

A refinement of the initial proposal can be obtained by replacing $\Gr^{\mathbb{C}}(\mathcal{D},n)$ with a Zariski open subset $\Gr^\mathrm{o}(\mathcal{D},n) \subset \Gr^{\mathbb{C}}(\mathcal{D},n)$ and $\mathcal{I}$ with its intersection $\mathcal{I}^\mathrm{o}$ with $\Gr^\mathrm{o}(\mathcal{D},n).$ This is explained in section 2.2. Then, it is possible to reformulate the strategy in terms of the topology and geometry of the quotient $\Gr^\mathrm{o}(\mathcal{D},n)/\mathcal{I}^\mathrm{o},$ and to interpret said reformulation as a linearization of the problem of deciding when an almost complex structure is integrable. This is mentioned in the concluding remarks, and will be the topic of future work.

The main goal of this article is to study the geometry of spaces involved in the aforementioned improvements of the initial strategy. Additionally, we provide a generalization of the embedding theory that was incepted by J-P.\ Demailly and H.\ Gaussier in \cite{DG}.

\textbf{Acknowledgment} I thank Jean-Pierre Demailly for meaningful discussions. I thank the European Research Council for financial support from the grant ``Algebraic and K{\"a}hler geometry" (ALKAGE, no.\ 670846). This project emerged at the CIME School on Non-K{\"a}hler geometry \cite{CIME}.

\section{Universal embedding spaces associated to \newline even-dimensional complex manifolds}

A \emph{complex directed manifold} is a pair $(X,\mathcal{D})$ of complex manifold $X$ and distribution $\mathcal{D} \subset T_X,$ i.e.\  $\mathcal{D}$ is a holomorphic sub-bundle of the tangent bundle of $X.$ Complex directed manifolds form a category whose morphisms are holomorphic maps $\Psi:X \rightarrow X'$ with $\Psi_{*}(\mathcal{D}) \subset \mathcal{D}'$ \cite{directed}. A morphism is \emph{\'{e}tale} if it is a local isomorphism. For example, an \'{e}tale morphism of real analytic manifolds is a real analytic map that is locally a diffeomorphism and an \'{e}tale morphism of complex manifolds is a holomorphic map that is locally a biholomorphism. If $(X,J)$ and $(X',J')$ are almost complex manifolds, a map $f:X \rightarrow X'$ is \emph{pseudo-holomorphic} provided that it satisfies the corresponding Cauchy-Riemann equation with $\bar{\partial}_{J,J'} f:=\frac{1}{2} (df + J' \circ df \circ J).$ If $W^{\mathbb{R}}$ is a real analytic manifold and $\dim_{\mathbb{R}}(W^{\mathbb{R}})=m,$ by the \emph{complexification of} $W^{\mathbb{R}},$ we mean the unique (germ of the) complex manifold $W$ that results from selecting an atlas on $W^{\mathbb{R}}$ and complexifying each of its real analytic transition functions near the real points. The complexification $W$ is such that $\dim_{\mathbb{C}}(Y)=m,$ and it possesses a natural anti-holomorphic involution whose set of fixed points is $W^{\mathbb{R}}.$

Let $k \geq n \geq 1$ and $Y$ be a complex manifold of complex dimension $2k.$ For every $y \in Y,$ consider the complex projective manifold of flags of signature $(k-n,k)$ in $T_{Y,y}$  

\begin{equation*}
\begin{split}
F_{(k-n,k)}(T_{Y,y})=&\{(S',\Sigma') \big| S' \subset \Sigma' \subset T_{Y,y} \mbox{ is a sequence of linear subspaces, }\\
&\dim_{\mathbb{C}}(S')=k-n\mbox{ and }\dim_{\mathbb{C}}(\Sigma')=k\}
\end{split}
\end{equation*} and the product manifold 

\[F^2_{(k-n,k)}(T_{Y,y})=\{(S',S'',\Sigma',\Sigma'') \big| (S',\Sigma'), (S'',\Sigma'') \in F_{(k-n,k)}(T_{Y,y})\}.\] Let \[Q_y=\{(S',S'',\Sigma',\Sigma'')\in F^2_{(k-n,k)}(T_{Y,y}) \big| \Sigma' \oplus \Sigma''=T_{Y,y} \}.\] Define \[\mathcal{Z}_n(Y):= \coprod_{y \in Y} Q_y.\] This is a complex manifold of complex dimension \[N_{n,k}:=2k+2(k^2+n(k-n)),\] and it bears a resemblance to twistor bundles and Grassmannians \cite{DG}. Let $\pi_Y:\mathcal{Z}_n(Y) \rightarrow Y$ be the projection map defined for any $y \in Y$ and $q_y \in Q_y$ by $\pi_Y(y,q_y)=y.$ 

Define $\mathbf{\Delta}_{n,k}$ to be the sub-bundle of $\pi^*_Y(T_Y)$ such that for any $w=(y,S',S'',\Sigma',\Sigma'') \in \mathcal{Z}_n(Y),$ we have $\mathbf{\Delta}_{n,k,w}=S' \oplus \Sigma''.$ Now define a distribution $\mathbf{D}_{n,k} \subset T_{\mathcal{Z}_n(Y)}$ by $\mathbf{D}_{n,k}:=d \pi^{-1}_Y(\mathbf{\Delta}_{n,k}).$ 

In the case $Y=\mathbb{C}^{2k},$ if we let \[Q:=\{(S',S'',\Sigma',\Sigma'') \in F^2_{(k-n,k)} (\mathbb{C}^{2k}) | \Sigma' \oplus \Sigma''=\mathbb{C}^{2k}\},\] then we simply get that $\mathcal{Z}_n(\mathbb{C}^{2k})=\mathbb{C}^{2k} \times Q.$ Therefore, the above construction recovers the one made by J-P.\ Demailly and H.\ Gaussier, who introduced the complex directed manifolds $(Z_{n,k},\mathcal{D}_{n,k}),$ where $Z_{n,k}$ is the complex, quasi-projective manifold of all $5$-tuples \[\{(z,S',S'',\Sigma',\Sigma'') \mid z \in \mathbb{C}^{2k}, (S',S'',\Sigma',\Sigma'') \in Q\}\] and $\mathcal{D}_{n,k}$ is the co-rank $n$ sub-bundle of $T_{Z_{n,k}},$ whose fiber at any $w \in Z_{n,k}$ is \[\mathcal{D}_{n,k,w}=\{(\zeta, u',u'',v',v'') \in T_{Z_{n,k},w} \big| \zeta \in S' \oplus \Sigma''\}\] (cf.\ Theorem 1.2 \cite{DG}). For $n$ fixed, the above defined complex directed manifold $(\mathcal{Z}_n(Y),\mathbf{D}_{n,k})$ will be called here the \emph{universal embedding space associated with} $Y,$ as we will soon see its universal property.

\begin{remark}\label{tri}
$\mathcal{Z}_n(Y) \xrightarrow{\pi_Y} Y$ is a (holomorphic) fiber bundle with typical fiber $Q,$ and if $Y=\mathbb{C}^{2k},$ the bundle is trivial. The universal embedding space $\mathcal{Z}_n(Y)$ is locally diffeomorphic to $\mathcal{Z}_n(\mathbb{C}^{2k}).$
\end{remark}

\begin{proof}
Given a holomorphic atlas $(U_{\alpha}, \psi_{\alpha})$ for $Y$ and $p \in U_{\alpha},$ the isomorphism $T_{Y,p} \simeq \mathbb{C}^{2k}$ induces a biholomorphism $Q_p \simeq_{q_{\alpha}} Q,$ and so $(U_{\alpha},Id_{U_{\alpha}} \times q_{\alpha})$ is a local trivialization.

Now observe that $\mathcal{Z}_n(U_{\alpha})=\pi^{-1}_Y(U_{\alpha}) \simeq U_{\alpha} \times Q \simeq \mathbb{C}^{2k} \times Q=\mathcal{Z}_n(\mathbb{C}^{2k}),$ where $\mathcal{Z}_n(U_{\alpha})=\coprod_{p \in U_{\alpha}} Q_p,$ where the first identification comes from the local trivialization and the second one, from the map $\psi_{\alpha}:U_{\alpha} \rightarrow \mathbb{C}^{2k}.$ 
\end{proof}

\subsection{Universal embedding property}

Next we construct embeddings of compact almost complex manifolds. Our approach is based on that used in the proof of Theorem 1.2 of \cite{DG}. Throughout, $(X,J_X)$ is a compact almost complex manifold of dimension $n \geq 1.$

\begin{proposition}\label{caliber}
Assume that there is a $C^{\infty}$ real embedding of $X$ into a real analytic $2k$-dimensional manifold $Y^{\mathbb{R}}.$ Assume further that $N_{X/Y^{\mathbb{R}}}$ admits a complex structure $J_N.$ Then, there is a totally real embedding $F:X \hookrightarrow \mathcal{Z}_n(Y)$ that is transverse to $\mathbf{D}_{n,k}$ and that induces the almost complex structure $J_X.$
\end{proposition}

\begin{proof}
Define $\tilde{J}:=J_X \oplus J_N,$ which is a complex structure on $T_{Y^{\mathbb{R}}}|_X \simeq T_X \oplus N_{X/{Y^{\mathbb{R}}}}.$ Consider the $J_Y$-complexification of $\tilde{J},$ $\tilde{J}^{\mathbb{C}}=J^{\mathbb{C}}_X \oplus J^{\mathbb{C}}_N:T_{Y}|_X \to T_{Y}|_X,$ and the sub-bundle $S:=\{0\} \oplus N^{\mathbb{C}}_{X/{Y^{\mathbb{R}}}}$ of $T_{Y}|_X.$ For any $x \in X$ and $(0,\eta_x) \in S_x=\{0\} \oplus N^{\mathbb{C}}_{X/{Y^{\mathbb{R}}}, x},$ $\tilde{J}^{\mathbb{C}}(x)(0,\eta_x)=(0, J^{\mathbb{C}}_N(\eta_x)) \in S_x,$ implying that $S_x$ is $\tilde{J}^{\mathbb{C}}(x)$-stable so that $\tilde{J}^{\mathbb{C}}(x)|_{S_x} \in \End(S_x).$ Let $\Sigma'_x$ be the $+i$ eigenspace for $\tilde{J}^{\mathbb{C}}(x)$ and $\Sigma''_x$ be the $-i$ eigenspace for $\tilde{J}^{\mathbb{C}}(x).$ Then, the $+i,$ respectively $-i,$ eigenspaces for $\tilde{J}^{\mathbb{C}}(x)|_{S_x}$ are $S'_x:=S_x \cap \Sigma'_x$ and $S''_x:=S_x \cap \Sigma''_x.$ More explicitly, these eigenspaces are $\Sigma'_x=T^{1,0}_{X,x} \oplus \Eig(J^{\mathbb{C}}_N (x), i),$ $\Sigma''_x=T^{0,1}_{X,x} \oplus \Eig(J^{\mathbb{C}}_N (x), -i),$ $S'_x=\{0\} \oplus \Eig(J^{\mathbb{C}}_N (x), i),$ and $S''_x=\{0\} \oplus \Eig(J^{\mathbb{C}}_N (x), -i).$ Note that $S'_x \subset \Sigma'_x,$ $S''_x \subset \Sigma''_x,$ $\Sigma'_x \oplus \Sigma''_x = T_Y|_{X,x},$ and $S_x =S'_x \oplus S''_x,$ where $\dim_{\mathbb{C}}(\Sigma'_x)=\dim_{\mathbb{C}}(\Sigma''_x)=\frac{1}{2} \dim_{\mathbb{C}}(T_Y|_{X,x})=k$ and $\dim_{\mathbb{C}}(S'_x)=\dim_{\mathbb{C}}(S''_x)=\frac{1}{2} \dim_{\mathbb{C}}(N^{\mathbb{C}}_{X/{Y^{\mathbb{R}},x}})=\frac{1}{2}(2k-2n)=k-n.$ Therefore, if $f:X \hookrightarrow Y^{\mathbb{R}}$ is the given $C^{\infty}$ real embedding, for any $x \in X$--in reality, $x$ belongs to $f(X)$--we have that $(S'_x,S''_x,\Sigma'_x,\Sigma''_x) \in Q_x=\{(S',S'',\Sigma',\Sigma'')\in F^2_{(k-n,k)}(T_{Y,x}) \mid \Sigma' \oplus \Sigma'' = T_{Y,x}\},$ and in this way, get an embedding $F:X \hookrightarrow \mathcal{Z}_n(Y),$ where $F(x)=(f(x),S'_x,S''_x,\Sigma'_x,\Sigma''_x).$

Let $\sigma:Y \rightarrow Y,$ $y \mapsto \overline{y},$ be an anti-holomorphic involution that realizes $Y^{\mathbb{R}}$ as its fixed point set. For 
$(S',S'',\Sigma',\Sigma'') \in Q_y,$ put $\overline{S'}:=d\sigma|_y(S') \subset T_{Y,\overline{y}},$ $\overline{\Sigma'}:=d \sigma|_y(\Sigma') \subset T_{Y,\overline{y}},$ and define $\overline{S''}$ and $\overline{\Sigma''}$ similarly. Then, $\sigma$ gives rise to the anti-holomorphic involution $\tilde{\sigma}:\mathcal{Z}_n(Y) \rightarrow \mathcal{Z}_n(Y),$ \[(y,S',S'',\Sigma',\Sigma'') \mapsto (\overline{y}, \overline{S''}, \overline{S'}, \overline{\Sigma''}, \overline{\Sigma'}).\] The real points $\mathcal{Z}_n(Y)^{\mathbb{R}}$ of $\mathcal{Z}_n(Y)$ are the fixed points of $\tilde{\sigma},$ so \[\mathcal{Z}_n(Y)^{\mathbb{R}}=\{(y,S',S'',\Sigma',\Sigma'') \in \mathcal{Z}_n(Y) \mid y \in Y^{\mathbb{R}}, S''=\overline{S'}, \Sigma''=\overline{\Sigma'}\}.\] The anti-holomorphic character of $\sigma$ implies that $d \sigma$ is type-reversing and point-wise conjugate linear, so $\overline{\Sigma'_x}=T^{0,1}_{X,x} \oplus \Eig(J^{\mathbb{C}}_N (x), -i)=\Sigma''_x$ and similarly, $\overline{S'_x}=S''_x.$ Therefore, $F(X) \subset \mathcal{Z}_n(Y)^{\mathbb{R}}.$ 

Since $\dim_{\mathbb{R}}(dF|_x(T_{X,x}))+\dim_{\mathbb{R}}(\mathbf{D}_{n,k,F(x)})=2N_{n,k}=\dim_{\mathbb{R}}(T_{\mathcal{Z}_n(Y),F(x)}),$ the embedding $F$ is transverse to $\mathbf{D}_{n,k,F(x)}$ if $dF|_x(T_{X,x})$ and $\mathbf{D}_{n,k,F(x)}$ intersect trivially. But the latter follows from $d \pi_Y|_{F(x)}(dF|_x(T_{X,x}) \cap \mathbf{D}_{n,k,F(x)})=df|_x(T_{X,x}) \cap S'_x \oplus \Sigma''_x=\{0\}.$ Hence $T_{\mathcal{Z}_n(Y),F(x)}=dF|_x(T_{X,x}) \oplus \mathbf{D}_{n,k,F(x)}.$

Let $J_{\mathcal{Z}_n(Y)}$ be the given complex structure on $\mathcal{Z}_n(Y).$ The quotient $T_{\mathcal{Z}_n(Y)}/{\mathbf{D}_{n,k}}$ is a holomorphic vector bundle on $\mathcal{Z}_n(Y).$ Since $\mathbf{D}_{n,k} \subset T_{\mathcal{Z}_n(Y)}$ is a holomorphic distribution, $J_{\mathcal{Z}_n(Y)}(\mathbf{D}_{n,k})=\mathbf{D}_{n,k}.$ So $J_{\mathcal{Z}_n(Y)}$ descends to a complex structure on $T_{\mathcal{Z}_n(Y)}/{\mathbf{D}_{n,k}}.$ The transversality of $F$ implies that at any $x \in X,$ there is a real isomorphism $\rho:T_{F(X),F(x)} \to T_{\mathcal{Z}_n(Y),F(x)}/{\mathbf{D}_{n,k,F(x)}}.$ Then, $J^{\mathcal{Z}_n(Y),\mathbf{D}_{n,k}}_{F(X)}(x):=\rho^{-1} \circ J_{\mathcal{Z}_n(Y)}(x) \circ \rho$ defines an almost complex structure $J^{\mathcal{Z}_n(Y),\mathbf{D}_{n,k}}_{F(X)}$ on $F(X),$ and then since $F$ is an embedding, the pullback section $J_F:=F^*(J^{\mathcal{Z}_n(Y),\mathbf{D}_{n,k}}_{F(X)})$ is an almost complex structure on $X.$ Note that \[T_{\mathcal{Z}_n(Y),F(x)}/{\mathbf{D}_{n,k,F(x)}} \simeq \Sigma'_x/S'_x,\] and so $T_{\mathcal{Z}_n(Y),F(x)}/{\mathbf{D}_{n,k,F(x)}}$ is isomorphic to the holomorphic tangent space $T^{1,0}_{X,x},$ which is $T_{X,x}$ endowed with the complex structure $J^{\mathbb{C}}_X(x).$ Run the above construction, with $T^{1,0}_{X,x}$ now playing the role of $T_{\mathcal{Z}_n(Y),F(x)}/{\mathbf{D}_{n,k,F(x)}},$ to find that $J_F=J_X.$
\end{proof}

\begin{lemma}\label{diagonal}
Let $M$ be a real analytic manifold of dimension $k > n,$ let $f:X \hookrightarrow M$ be a $C^{\infty}$ real embedding, and $i_{\Delta}:M \hookrightarrow M \times M$ be the diagonal embedding $i_{\Delta}(x)=(x,x).$ Embed $X$ into $M \times M$ via $i_{\Delta} \circ f.$ Then, the normal bundle $N_{X/{M \times M}}$ has a natural complex structure $J_N.$
\end{lemma}

\begin{proof}
Since $T_M \simeq N_{M/{M \times M}},$ we obtain the following real analytic splittings \[N_{X/{M \times M}} \simeq N_{X/M} \oplus N_{M/{M \times M}}|_X \simeq N_{X/M} \oplus T_{M}|_X \simeq N_{X/M} \oplus N_{X/M} \oplus T_X.\]  
 
Let $J_{N_{X/M} \oplus N_{X/M}}$ be the tautological complex structure that is given by $J_{N_{X/M} \oplus N_{X/M}}(\zeta,\eta)=(-\eta,\zeta).$ Put $J_N:=J_{N_{X/M} \oplus N_{X/M}} \oplus (-J_X),$ which defines a complex structure on $N_{X/{M \times M}}.$
\end{proof}

If we take $Y^{\mathbb{R}}=M \times M$ as in the lemma, we then have that
\[\Sigma'_x=T^{1,0}_{X,x} \oplus \{(u,-iu) \mid u \in N^{\mathbb{C}}_{X/M,x}\} \oplus T^{0,1}_{X,x},\] \[\Sigma''_x=T^{0,1}_{X,x} \oplus \{(u,iu) \mid u \in N^{\mathbb{C}}_{X/M,x}\} \oplus T^{1,0}_{X,x},\] \[S'_x=\{0\} \oplus \{(u,-iu) \mid u \in N^{\mathbb{C}}_{X/M,x}\} \oplus T^{0,1}_{X,x}, \mbox{ and}\] \[S''_x=\{0\} \oplus \{(u,iu) \mid u \in N^{\mathbb{C}}_{X/M,x}\} \oplus T^{1,0}_{X,x}.\] Here we can see why it is nicer to take $-J_X$ in the definition of $J_N,$ as otherwise $\Sigma'_x$ and $\Sigma''_x$ would have repeated direct sum factors of $T^{1,0}_{X,x}$ and $T^{0,1}_{X,x}.$ 

\begin{corollary}\label{corona}
If $n \geq 1$ and $k \geq 2n,$ any compact almost complex $n$-dimensional manifold $(X,J_X)$ admits a universal embedding $F:(X,J_X) \hookrightarrow (\mathcal{Z}_n(Y),\mathbf{D}_{n,k}),$ where $\dim_{\mathbb{R}}(Y^{\mathbb{R}})=2k.$  
\end{corollary}

\begin{proof}
Take any real analytic manifold $M$ of dimension $k$ and put $Y^{\mathbb{R}}=M \times M.$ Let $\psi:X^{2n} \hookrightarrow \mathbb{R}^{4n}$ be a Whitney embedding and $\phi:\mathbb{R}^{4n} \rightarrow B$ be a diffeomorphism of $\mathbb{R}^{4n}$ with some open ball $B \subset Y^{\mathbb{R}}.$ Then, $\phi \circ \psi: X \hookrightarrow Y^{\mathbb{R}}$ is a real analytic embedding. The proposition builds $F:X \hookrightarrow (\mathcal{Z}_n(Y),\mathbf{D}_{n,k})$ out of the embedding $\phi \circ \psi.$
\end{proof}

We call an embedding $F$ as in Corollary \ref{corona}
a \emph{universal embedding} and denote it by $F:(X,J_X) \hookrightarrow (\mathcal{Z}_n(Y),\mathbf{D}_{n,k}).$ 

\begin{remark}\label{sphere6}
Consider the $6$-dimensional sphere $S^6$ equipped with the octonion almost complex structure $J_{\mathbb{O}}.$ Proposition \ref{caliber} manufactures a universal embedding $F:(S^6,J_{\mathbb{O}}) \hookrightarrow (\mathcal{Z}_3 (\mathbb{C}^8),\mathcal{D}_{3,4})$ from the inclusion mapping $i:S^6 \hookrightarrow Im(\mathbb{O}) \subset \mathbb{R}^8,$ where $Im(\mathbb{O})\simeq \mathbb{R}^7$ are the imaginary octonions.
\end{remark}

\subsection{Functorial property with respect to \'{e}tale morphisms}
For each $n \in \mathbb{N},$ let $\mathbb{C}\Man^n_{Et}$ be the category of complex manifolds of dimension $n$ whose morphisms are \'{e}tale morphisms. For each $k \geq n \geq 1,$ we have functors $\mathbf{Z}_{n}: \mathbb{C}\Man^{2k}_{Et} \rightarrow \mathbb{C}\Man^{N_{n,k}}_{Et},$ given by $\mathbf{Z}_n(Y)=\mathcal{Z}_n(Y),$ and for any {\'e}tale morphism $f:Y \rightarrow Y',$ $\mathbf{Z}_n(f):\mathcal{Z}_n(Y) \rightarrow \mathcal{Z}_{n}(Y')$ is the map that is defined at each $w=(y,S',S'',\Sigma',\Sigma'') \in \mathcal{Z}_n(Y)$ by \[\mathbf{Z}_n(f)(w)=\big(f(y), df|_y(S'), df|_y(S''), df|_y(\Sigma'), df|_y(\Sigma'')\big),\] where $df|_y(S'):=\im(df(y)|_{S'}),$ $df|_y(S''):=\im(df(y)|_{S''}),$ and so forth. Throughout, we will write $\mathcal{Z}_n(f)$ instead of $\mathbf{Z}_n(f).$

For a fixed $n,$ \emph{the universal embedding space associated with} $Y$ is a complex directed manifold $(\mathcal{Z}_n(Y),\mathbf{D}_{n,k})$ such that every compact almost complex $n$-dimensional manifold $(X,J_X)$ admits a totally real, transverse to $\mathbf{D}_{n,k},$ $J_X$-inducing embedding $F:X \hookrightarrow \mathcal{Z}_n(Y),$ and it satisfies the following universal property. For any $n$-dimensional compact almost complex manifold $(X',J_{X'}),$ any $C^{\infty}$ real embeddings $g:X \hookrightarrow Y^{\mathbb{R}}$ and $g':X' \hookrightarrow {Y'}^{\mathbb{R}},$ any pseudo-holomorphic \'{e}tale map $\psi_X:(X,J_X)\rightarrow (X',J_{X'}),$ and any \'{e}tale morphism $\psi_Y:Y^{\mathbb{R}} \rightarrow Y'^{\mathbb{R}}$ that fit into a commutative diagram
\[
\begin{tikzcd}
  (X,J_X) \arrow[r, hookrightarrow, "g"] \arrow[d,"\psi_X"]
    & Y^{\mathbb{R}} \arrow[d, "\psi_Y"] \\
  (X',J_{X'}) \arrow[r, hookrightarrow, "g'"]
&Y'^{\mathbb{R}} , \end{tikzcd}\]
there is a corresponding functorially defined morphism $\mathcal{Z}_n(\psi_Y):(\mathcal{Z}_{n}(Y),\mathbf{D}_{n,k}) \rightarrow (\mathcal{Z}_{n}(Y'),\mathbf{D}'_{n,k})$ of complex directed manifolds making the diagram 
\[
\begin{tikzcd}
  (X,J_X) \arrow[r, hookrightarrow, "F"] \arrow[d,"\psi_X"]
    &\mathcal{Z}_{n}(Y) \arrow[d, "\exists \mathcal{Z}_n(\psi_Y)", dashed] \\
 (X',J_{X'}) \arrow[r, hookrightarrow, "F'"]
&\mathcal{Z}_{n}(Y'), \end{tikzcd}\]
commute. By the Nash-Tognoli Theorem, we may assume that $X$ and $X'$ are smooth real algebraic varieties and that $g:X \hookrightarrow Y^{\mathbb{R}}$ and $g':X'\hookrightarrow Y'^{\mathbb{R}}$ are algebraic. Then, our construction gives that $F$ and $F'$ are real algebraic as well.

\section{The geometry of $\mathcal{Z}_n(Y)$ and related bundles}

\subsection{Coordinates on $\mathcal{Z}_n(Y)$}

Let $U_y \simeq_{\psi} \mathbb{C}^{2k}$ be any holomorphic coordinate chart that is centered at a given point $y \in Y.$ Let $p \in U_y.$ We write $\psi(p)=(x_1,\dots,x_{2k})$ for the holomorphic coordinates of the point $p.$ Let $S'_p=\Span_{\mathbb{C}}(\frac{\partial}{\partial x_j}\big|_p)_{j=n+1}^k,$ $S''_p=\Span_{\mathbb{C}}(\frac{\partial}{\partial x_j}\big|_p)_{j=n+k+1}^{2k},$ $\Sigma'_p=\Span_{\mathbb{C}}(\frac{\partial}{\partial x_j}\big|_p)_{j=1}^{k},$ $\Sigma''_p=\Span_{\mathbb{C}}(\frac{\partial}{\partial x_j}\big|_p)_{j=k+1}^{2k},$ and put  $f_p:=(S'_p,S''_p,\Sigma'_p,\Sigma''_p).$ From this point on, we will omit reference to the point $p$ in the vector $\frac{\partial}{\partial x_j}\big|_p,$ and denote it simply by $\frac{\partial}{\partial x_j}.$ We develop a coordinate chart $U_y \times \mathcal{A}(f_y) \simeq \mathbb{C}^{2k} \times \mathbb{C}^{N_{n,k}-2k}=\mathbb{C}^{N_{n,k}}$ that is centered at $w_y:=(y,f_y).$ To that end, let $E_{S'_p}=\Span_{\mathbb{C}}(\frac{\partial}{\partial x_j})_{j=1}^n$ and $E_{S''_p}=\Span_{\mathbb{C}}(\frac{\partial}{\partial x_j})_{j=k+1}^{n+k}.$ In addition to $\Sigma'_p \oplus \Sigma''_p = \T_{Y,p},$ there are direct sum decompositions $S'_p \oplus E_{S'_p}=\Sigma'_p$ and $S''_p \oplus E_{S''_p}=\Sigma''_p.$ Define 

\begin{equation*}
\begin{split}
\mathcal{A}(f_y):= & \Big\{(S',S'',\Sigma', \Sigma'') \in Q_p \mid p \in U_y, S' \cap E_{S'_p}=\{0\}, S'' \cap E_{S''_p}=\{0\},\\
& \Sigma' \cap \Sigma''_p=\{0\}, \Sigma'' \cap \Sigma'_p=\{0\} \Big\}.
\end{split}
\end{equation*}

For any $(S',S'',\Sigma', \Sigma'') \in \mathcal{A}(f_y),$ $S',$ $S'',$ $\Sigma'$ and $\Sigma''$ correspond uniquely to maps $f_{S'} \in \Hom(S'_p,E_{S'_p}), f_{S''} \in \Hom(S''_p,E_{S''_p}), f_{\Sigma'} \in \Hom(\Sigma'_p, \Sigma''_p)$ and $f_{\Sigma''} \in \Hom(\Sigma''_p,\Sigma'_p),$ respectively, in the sense that $S'=\Gamma(f_{S'}),$ $S''=\Gamma(f_{S''}),$ $\Sigma'=\Gamma(f_{\Sigma'}),$ and $\Sigma''=\Gamma(f_{\Sigma''}),$ where $\Gamma(g)$ denotes the graph of the function $g;$ e.g. $\Gamma(f_{S'})=\{x + f_{S'}(x) | x \in S'_p\}.$ Suppose that

\[f_{S'}(\frac{\partial}{\partial x_j})=\sum_{i=1}^n z_{ij}\frac{\partial}{\partial x_i}, \mbox{ for }n+1 \leq j \leq k,\]

\[f_{S''}(\frac{\partial}{\partial x_j})=\sum_{i=k+1}^{n+k} z_{ij}\frac{\partial}{\partial x_i}, \mbox{ for }n+k+1 \leq j \leq 2k,\]

\[f_{\Sigma'}(\frac{\partial}{\partial x_j})=\sum_{i=k+1}^{2k} z_{ij}\frac{\partial}{\partial x_i} \mbox{ for }1 \leq j \leq k, \mbox{ and }\]

\[f_{\Sigma''}(\frac{\partial}{\partial x_j})=\sum_{i=1}^k z_{ij}\frac{\partial}{\partial x_i},\mbox{for }k+1 \leq j \leq 2k.\]

Then, 

\[S'=\Span_{\mathbb{C}}\big(\frac{\partial}{\partial x_j}+\sum_{i=1}^n z_{ij}\frac{\partial}{\partial x_i}\big)_{j=n+1}^k, S''=\Span_{\mathbb{C}}\big(\frac{\partial}{\partial x_j} + \sum_{i=k+1}^{n+k} z_{ij}\frac{\partial}{\partial x_i}\big)_{j=n+k+1}^{2k},\] 

\[\Sigma'=\Span_{\mathbb{C}}\big(\frac{\partial}{\partial x_j} +\sum_{i=k+1}^{2k} z_{ij}\frac{\partial}{\partial x_i}\big)_{j=1}^k \mbox{ and }\Sigma''=\Span_{\mathbb{C}}\big(\frac{\partial}{\partial x_j} + \sum_{i=1}^k z_{ij}\frac{\partial}{\partial x_i}\big)_{j=k+1}^{2k}.\]

This defines a coordinate map $q:\mathcal{A}(f_y) \rightarrow \mathbb{C}^{2(k^2+n(k-n))}=\mathbb{C}^{N_{n,k}-2k},$ \[q(S',S'',\Sigma',\Sigma'')=Z:=\left(
\begin{array}{c|c}
Z_{S'} & Z_{\Sigma''} \\
\hline
Z_{\Sigma'} & Z_{S''}
\end{array}
\right),\] where \[Z_{S'}=\left(
\begin{array}{ccc}
\mathbf{0_{n \times n}} & \vline & z_{ij} \\
\hline 
&\mathbf{0_{(k-n) \times k}}&
\end{array}
\right)_{1 \leq i \leq n, n+1 \leq j \leq k},\] 

\[Z_{S''}=\left(
\begin{array}{ccc}
\mathbf{0_{n \times n}} & \vline & z_{ij} \\
\hline 
&\mathbf{0_{(k-n) \times k}}&
\end{array}
\right)_{k+1 \leq i  \leq n+k , n+k+1 \leq j \leq 2k},\] 

\[Z_{\Sigma'}=\left(
\begin{array}{c}
z_{ij}
\end{array}
\right)_{k+1 \leq i \leq 2k, 1 \leq j \leq k},\] and 

\[Z_{\Sigma''}=\left(
\begin{array}{c}
z_{ij}
\end{array}
\right)_{1 \leq i \leq k, k+1 \leq j \leq 2k},\] and note that $q(f_y)=0.$

Coordinates centered at $w_y$ are then given by the map $\phi:=\psi \times q:U_y \times \mathcal{A}(f_y) \rightarrow \mathbb{C}^{N_{n,k}},$ $\phi(y,S',S'',\Sigma',\Sigma'')=(\psi(y), q(S',S'',\Sigma',\Sigma''))=(\psi(y),Z),$ and note that $\phi(w_y)=0.$     

\subsection{Sub-bundles of the Grassmannian bundle}

The torsion operator is the section \[\mathbf{\theta} \in H^0\big(\mathcal{Z}_n(Y), \mathcal{O}(\Lambda^2(\mathbf{D}^*_{n,k}) \otimes T_{\mathcal{Z}_n(Y)}/\mathbf{D}_{n,k})\big),\] where \[\mathbf{\theta}(w): \mathbf{D}_{n,k,w} \times \mathbf{D}_{n,k,w} \rightarrow T_{\mathcal{Z}_n(Y),w}/\mathbf{D}_{n,k,w},\] $\theta(w)(\zeta,\eta)=[\zeta,\eta] \mod{\mathbf{D}_{n,k,w}}.$ At the central point, we have the coordinate form

\begin{lemma}\label{torsion}
\[\theta(w_y)=-2 \sum_{j=n+1}^{2k} \sum_{i=1}^n dx_j \wedge dz_{ij} \otimes \frac{\partial}{\partial x_i}.\]
\end{lemma}

\begin{proof}
Let $\mathcal{I}:=\{1,\dots,n\} \times \{n+1,\dots,k\} \cup \{k+1,\dots,n+k\} \times \{n+k+1,\dots,2k\} \cup \{k+1,\dots,2k\} \times \{1,\dots,k\} \cup \{1,\dots,k\} \times\{k+1,\dots,2k\}$ and define 

\begin{equation*}
\begin{split}
(\mathcal{I}):= & \big((1,n+1),\dots,(1,k),\dots,(n,n+1),\dots,(n,k),(k+1,n+k+1),\dots,(k+1,2k),\\
& \dots,(n+k,n+k+1),\dots,(n+k,2k), (k+1,1)\dots,(k+1,k),\dots,(2k,1),\\
&\dots,(2k,k),(1,k+1),\dots,(1,2k),\dots,(k,k+1),\dots,(k,2k)\big).
\end{split}
\end{equation*} 

Here we work with the less compact, but equivalent coordinates $\phi(p,S',S'',\Sigma',\Sigma'')=(x, (z_{ij})_{(i,j) \in (\mathcal{I})}).$ Let $w=(p,S',S'',\Sigma', \Sigma'') \in U_y \times \mathcal{A}(f_y)$ so that 

\[\Big(\Big(\frac{\partial}{\partial x_l} \Big)_{l=1}^{2k}, \Big(\frac{\partial}{\partial z_{ij}}\Big)_{(i,j) \in (\mathcal{I})}\Big)\] is a basis of the tangent space $T_{\mathcal{Z}_{n}(Y),w}.$ Then,

\[\mathbf{D}_{n,k,w_y}=\Big\{(\zeta, u',u'',v',v'') \in T_{\mathcal{Z}_{n}(Y),w_y} \mid \zeta \in S'_y \oplus \Sigma''_y=\Span_{\mathbb{C}}\Big(\frac{\partial}{\partial x_l}\Big)_{l=n+1,\dots,2k}\Big\} \mbox{ and }\]

\[\mathbf{D}_{n,k,w}=\{(\zeta, u',u'',v',v'') \in T_{\mathcal{Z}_{n}(Y),w} \mid \zeta \in S' \oplus \Sigma''\}.\] 

Let $a:U_y \times \mathcal{A}(f_y) \rightarrow \M_{n \times (N_{n,k}-n)}(\mathbb{C})$ be the function $a(w)=(\chi_{\mathcal{I}} z_{ij}),$ where 

\[\chi_{\mathcal{I}}=\left\{
\begin{array}{ll}
      1 &\mbox{ if } (i,j) \in \mathcal{I} \\
      0 & \mbox{ otherwise.} \\
\end{array} 
\right. \]

Then, since $\chi_{\mathcal{I}}=0$ on $\{n+1,\dots,k\} \times \{n+1,\dots,k\},$

\[S' \oplus \Sigma''=\Span_{\mathbb{C}}\Big(\frac{\partial}{\partial x_j} + \sum_{i=1}^k a_{ij}(w) \frac{\partial}{\partial x_i} \Big)_{j=n+1,\dots,2k}.\] 

Therefore,

\[\mathbf{D}_{n,k,w_y} =\Span_{\mathbb{C}}\Big(\frac{\partial}{\partial x_l}\Big)_{l=n+1,\dots,2k}\oplus \Span_{\mathbb{C}} \Big(\frac{\partial}{\partial z_{ij}}\Big)_{(i,j) \in (\mathcal{I})},\]

\[\mathbf{D}_{n,k,w} =\Span_{\mathbb{C}}\Big(\frac{\partial}{\partial x_j} + \sum_{i=1}^k a_{ij}(w) \frac{\partial}{\partial x_i} \Big)_{j=n+1,\dots,2k} \oplus \Span_{\mathbb{C}}\Big(\frac{\partial}{\partial z_{ij}}\Big)_{(i,j) \in (\mathcal{I})},\]

\[\mbox{ and so }T_{\mathcal{Z}_{n}(Y),w_y}/\mathbf{D}_{n,k,w_y} \simeq \Span_{\mathbb{C}} \Big(\frac{\partial}{\partial x_l}\Big)_{l=1,\dots,n}.\]

Note that $a(w_y)=0.$ At $w_y,$ 

\begin{equation*}
\begin{split}
\Big[\frac{\partial}{\partial x_j} + \sum_{i=1}^k a_{ij} \frac{\partial}{\partial x_i},\frac{\partial}{\partial x_l} + \sum_{i=1}^k a_{il} \frac{\partial}{\partial x_i}  \Big] &=\sum_{i=1}^k\Big(\frac{\partial a_{il}}{\partial x_j} - \frac{\partial a_{ij}}{\partial x_l}\Big) \frac{\partial}{\partial x_i} \\
&=0
\end{split}
\end{equation*} 

because the function $a$ is independent of $x_i$ for all $i,$ and 

\begin{equation*}
\begin{split}
\theta(w_y)\Big(\frac{\partial}{\partial x_j} + \sum_{i=1}^k a_{ij} \frac{\partial}{\partial x_i}, \frac{\partial}{\partial z_{ml}} \Big) &=\sum_{i=1}^k \theta(w_y)\Big(a_{ij} \frac{\partial}{\partial x_i}, \frac{\partial}{\partial z_{ml}} \Big)\\
&=\sum_{i=1}^k \Big(-\frac{\partial a_{ij}}{\partial z_{ml}} \frac{\partial}{\partial x_i} \mod{\mathbf{D}_{n,k,w_y}}\Big) \\
&=\sum_{i=1}^n -\frac{\partial a_{ij}}{\partial z_{ml}}  \frac{\partial}{\partial x_i}.
\end{split}
\end{equation*} 

Then since, $1 \leq i \leq n,$ $n+1 \leq j \leq 2k,$ and $(m,l) \in \mathcal{I},$ the index domains (i.e.\ the domains of $i,j,l$ and $m$) overlap exactly at $\{1,\dots,n\} \times \{n+1,\dots,2k\} \cap \mathcal{I}=\{1,\dots,n\} \times \{n+1,\dots,2k\}.$ Therefore,  

\begin{equation*}
\begin{split}
\theta(w_y)&=\sum_{j,l=n+1}^{2k} \sum_{m=1}^n \theta(w_y)\Big(\frac{\partial}{\partial x_j} + \sum_{i=1}^k a_{ij} \frac{\partial}{\partial x_i}, \frac{\partial}{\partial z_{ml}} \Big) dx_j \wedge dz_{ml} + \\
&\sum_{j,l=n+1}^{2k} \sum_{m=1}^n \theta(w_y)\Big(\frac{\partial}{\partial z_{ml}}, \frac{\partial}{\partial x_j} + \sum_{i=1}^k a_{ij} \frac{\partial}{\partial x_i}\Big)dz_{ml} \wedge dx_j \\
&=\sum_{j,l=n+1}^{2k} \sum_{m=1}^n \Big( \sum_{i=1}^n -\frac{\partial a_{ij}}{\partial z_{ml}} \frac{\partial}{\partial x_i} \Big) dx_j \wedge dz_{ml} + \\
&\sum_{j,l=n+1}^{2k} \sum_{m=1}^n \Big(\sum_{i=1}^n \frac{\partial a_{ij}}{\partial z_{ml}} \frac{\partial}{\partial x_i} \Big)dz_{ml} \wedge dx_j \\
&=-2\sum_{j,l=n+1}^{2k} \sum_{i,m=1}^n \frac{\partial a_{ij}}{\partial z_{ml}} dx_j \wedge dz_{ml} \otimes \frac{\partial}{\partial x_i} \\
&=-2 \sum_{j=n+1}^{2k} \sum_{i=1}^n dx_j \wedge dz_{ij} \otimes \frac{\partial}{\partial x_i}.
\end{split}
\end{equation*} 

\end{proof}

\begin{remark}\label{the}
For any $X=\sum_{l=n+1}^{2k} X_l \frac{\partial}{\partial x_l} + \sum_{(i,j) \in (\mathcal{I})} X_{ij} \frac{\partial}{\partial z_{ij}}, Y=\sum_{l=n+1}^{2k} Y_l \frac{\partial}{\partial x_l} + \sum_{(i,j) \in (\mathcal{I})} Y_{ij} \frac{\partial}{\partial z_{ij}} \in \mathbf{D}_{n,k,w_y},$

\[\theta(w_y)(X,Y)=-2\sum_{i=1}^n \sum_{j=n+1}^{2k} (X_j Y_{ij}-Y_jX_{ij})\frac{\partial}{\partial x_i}.\]
\end{remark}

Recall the projection mapping $\pi_Y:\mathcal{Z}_n(Y) \rightarrow Y,$ $\pi_Y(y,q_y)=y.$ If $Y=\mathbb{C}^{2k},$ we simply denote this map by $\pi.$ Consider the Grassmannian bundle $\mathbf{\Gr}^{\mathbb{C}}(\mathbf{D}_{n,k},n)$ on $\mathcal{Z}_n(Y),$ whose fiber at $w \in \mathcal{Z}_n(Y)$ is the Grassmannian $\mathbf{\Gr}^{\mathbb{C}}(\mathbf{D}_{n,k,w},n).$ For any $w \in \mathcal{Z}_n(Y),$ define \[\Gro^\mathrm{o}_{n,k,w}:=\{S \in \Gr^{\mathbb{C}}(\mathbf{D}_{n,k,w},n) \mid d\pi_Y(w)|_S \mbox{ is injective}\}\] and \[\mathbf{I}_{n,k,w}:=\{S \in \Gr^{\mathbb{C}}(\mathbf{D}_{n,k,w},n) \mid \mathbf{\theta}(w)|_{S \times S}=0\}.\] Let $\Gro^\mathrm{o}_{n,k}$ be the sub-bundle of the Grassmannian bundle with fiber $\Gro^\mathrm{o}_{n,k,w}$ at $w,$ and define $\mathbf{I}_{n,k}$ similarly. Now let $\mathbf{I}^\mathrm{o}_{n,k}$ be the sub-bundle whose fiber over $w$ is $\mathbf{I}^\mathrm{o}_{n,k,w}:=\Gro^\mathrm{o}_{n,k,w} \cap \mathbf{I}_{n,k,w}.$ When $Y=\mathbb{C}^{2k},$ we denote these bundles by $\Gr^\mathrm{o}_{n,k}, \mathcal{I}_{n,k},$ and $\mathcal{I}^\mathrm{o}_{n,k}.$ 

The refinement of the strategy that was mentioned in the introduction is now obtained by replacing the Grassmannian bundle and the isotropic locus $\mathcal{I}_{n,k}$ with the Zariski open $\Gr^\mathrm{o}_{n,k} \subset \mathbf{\Gr}^{\mathbb{C}}(\mathcal{D}_{n,k},n)$ and with $\mathcal{I}^\mathrm{o}_{n,k},$ respectively. Indeed, such a replacement makes sense because if $F:(X,J_X) \hookrightarrow \mathcal{Z}_n(\mathbb{C}^{2k})^{\mathbb{R}}$ is a universal embedding, the composition $d\pi|_{F(x)} \circ \bar{\partial}_{J_X} F|_x$ must be injective. So $\im(\bar{\partial}_{J_X} F) \subset \Gr^\mathrm{o}_{n,k},$ and therefore if $J_X$ is integrable, $\im(\bar{\partial}_{J_X} F) \subset \mathcal{I}^\mathrm{o}_{n,k}.$ So we are now interested in the topology of $\mathcal{I}^\mathrm{o}_{n,k}$ relative to $\Gr^\mathrm{o}_{n,k}.$ Note that if we start with a universal embedding $F:(X,J_X) \hookrightarrow \mathcal{Z}_n(Y),$ we can state the strategy in terms of $\Gro^\mathrm{o}_{n,k}$ and $\mathbf{I}^\mathrm{o}_{n,k}.$ In that case, the homotopy of interest would be a map $\im(\bar{\partial}_{J_t} F_t (x)):[0,1] \times X \rightarrow \Gro^\mathrm{o}_{n,k}.$ At the moment, we do not know if there are any advantages to working with this version of the strategy.

\subsection{A functorial group action}
This is a brief overview of basic facts that we will use in future work on understanding the geometry of $\mathcal{Z}_n(Y)$ and related bundles via group actions.

Let $(e_j)_{j=1}^{2k}$ be the standard basis of $\mathbb{C}^{2k},$ $S'_0=\Span_{\mathbb{C}}(e_j)_{j=n+1}^k, S''_0=\Span_{\mathbb{C}}(e_j)_{j=n+k+1}^{2k},$ $\Sigma'_0=\Span_{\mathbb{C}}(e_j)_{j=1}^{k},$ and $\Sigma''_0=\Span_{\mathbb{C}}(e_j)_{j=k+1}^{2k}.$ Put $f_0:=(S'_0,S''_0,\Sigma'_0,\Sigma''_0) \in F^2_{(k-n,k)}(\mathbb{C}^{2k})$ and $w_0:=(0,f_0).$ 

\begin{lemma}
The group $\GL_{2k}(\mathbb{C})$ acts transitively on $Q$ and the stabilizer $\Lambda$ of $f_0$ is the subgroup of $\GL_{2k}(\mathbb{C})$ of all matrices of the form \[\left(
\begin{array}{c|c}
B_{k \times k} & 0_{k \times k} \\
\hline
0_{k \times k} & B'_{k \times k}
\end{array}
\right),\] where \[B_{k \times k}=\left(
\begin{array}{ccc}
\mathbf{B_{n \times n}} &\vline& \mathbf{0_{n \times (k-n)}}\\
\hline
& \mathbf{B_{(k-n) \times k}}& 
\end{array}
\right) \in \GL_{k}(\mathbb{C})\] and \[B'_{k \times k}=\left(
\begin{array}{ccc}
\mathbf{B'_{n \times n}} &\vline& \mathbf{0_{n \times (k-n)}}\\
\hline
& \mathbf{B'_{(k-n) \times k}}& 
\end{array}
\right) \in \GL_{k}(\mathbb{C}).\] As a result, $Q \simeq \GL_{2k}(\mathbb{C})/\Lambda.$ 
\end{lemma}

\begin{proof}
For any subspace $S=\Span_{\mathbb{C}}(s_j)_{j=1}^r$ of $\mathbb{C}^{2k},$ define $GS:=\Span_{\mathbb{C}}(Gs_j)_{j=1}^r,$ for any $G \in \GL_{2k}(\mathbb{C}).$ The action of $\GL_{2k}(\mathbb{C})$ on $Q,$ given by $(G,(S',S'',\Sigma',\Sigma'')) \mapsto G(S,S'',\Sigma',\Sigma''),$ where $G(S,S'',\Sigma',\Sigma'')=(GS',GS'',G\Sigma',G\Sigma''),$ is transitive. To compute the stabilizer of $f_0,$ let $B=\begin{pmatrix}B_{ij} \end{pmatrix}_{1 \leq i,j \leq 2k} \in \GL_{2k}(\mathbb{C})$ and note that $Bf_0=f_0$ iff 
\[\mbox{for each } n+1 \leq j \leq k, B_{ij} = 0 \mbox{ for all } 1 \leq i \leq n \mbox{ and } k+1 \leq i \leq 2k,\]
\[\mbox{for each } n+k+1 \leq j \leq 2k, B_{ij} = 0 \mbox{ for all } k+1 \leq i \leq n+k,\]
\[\mbox{for each } 1 \leq j \leq k, B_{ij} = 0 \mbox{ for all } k+1 \leq i \leq 2k,\] and \[\mbox{for each } k+1 \leq j \leq 2k, B_{ij} = 0 \mbox{ for all }1 \leq i \leq k.\] Therefore, $\begin{pmatrix}B_{ij} \end{pmatrix}_{1 \leq i \leq n,n+1 \leq j \leq k}=0_{n \times (k-n)},$ $\begin{pmatrix}B_{ij} \end{pmatrix}_{k+1 \leq i \leq n+k,n+k+1 \leq j \leq 2k}=0_{n \times (k-n)},$ $\begin{pmatrix}B_{ij} \end{pmatrix}_{k+1 \leq i \leq 2k,1 \leq j \leq k}=0_{k \times k},$ and $\begin{pmatrix}B_{ij} \end{pmatrix}_{1 \leq i \leq k,k+1 \leq j \leq 2k}=0_{k \times k}$ so that $B$ is as claimed. 
\end{proof}

The group $\Aut^{\hol}(Y)$ of biholomorphisms of $Y$ acts functorially on $\mathcal{Z}_n(Y)$: \[\Aut^{\hol}(Y) \times \mathcal{Z}_n(Y) \rightarrow \mathcal{Z}_n(Y),\] \[(f,(y,S',S'',\Sigma',\Sigma'')) \mapsto \mathcal{Z}_n(f)(y,S',S'',\Sigma',\Sigma''),\] where recall $\mathcal{Z}_n(f)(y,S',S'',\Sigma',\Sigma'')= (f(y),df|_y(S',S'',\Sigma',\Sigma'')).$ Of course it can happen that $\Aut^{\hol}(Y)=\{Id_Y\}.$ Although this action is generally non-transitive, there are exceptions, for example when $Y=\mathbb{C}^{2k}.$ Consider the subgroup $\Aff(\mathbb{C}^{2k})$ of $\Aut^{\hol}(\mathbb{C}^{2k}).$ For any $f=Bz + c \in \Aff(\mathbb{C}^{2k}),$ note that \[\mathcal{Z}_n(f)(y,S',S'',\Sigma',\Sigma'')=(By+c,df|_y(S',S'',\Sigma',\Sigma''))=(By+c,B(S',S'',\Sigma',\Sigma'')).\]

\begin{lemma}\label{tt}
The functorial action of $\Aff(\mathbb{C}^{2k})$ on $\mathcal{Z}_n(\mathbb{C}^{2k})$ is transitive, the stabilizer of $w_0$ is $\mathcal{L}:=\{f=Bz \mid B \in \Lambda\},$ and $\mathcal{Z}_n(\mathbb{C}^{2k}) \simeq \Aff(\mathbb{C}^{2k}) /\mathcal{L}.$ 
\end{lemma}

\begin{proof}
Let $(y,S',S'',\Sigma',\Sigma'') \in \mathcal{Z}_n(\mathbb{C}^{2k})=\mathbb{C}^{2k} \times Q.$ We saw that $\GL_{2k}(\mathbb{C})$ acts transitively on $Q.$ So there is a $B_0 \in \GL_{2k}(\mathbb{C})$ such that $(S',S'',\Sigma',\Sigma'')=B_0 f_0.$ Put $c_0=y$ so that $\mathcal{Z}_n(B_0z + c_0)(w_0)=(y,S',S'',\Sigma',\Sigma'').$ 
\end{proof}

\section{Affine bundle structure}

At the beginning of Section 1, we introduced the sub-bundle $\mathbf{\Delta}_{n,k} \subset \pi^*_Y(T_Y),$ where for any $w=(y,S',S'',\Sigma',\Sigma'') \in \mathcal{Z}_n(Y),$ $\mathbf{\Delta}_{n,k,w}=S' \oplus \Sigma'',$ and recall that $\mathbf{D}_{n,k}=d \pi^{-1}_Y(\mathbf{\Delta}_{n,k}).$ Let $\pi_{n,k}:\Gr^{\mathbb{C}}(\mathbf{\Delta}_{n,k},n) \rightarrow \mathcal{Z}_n(Y)$ be the Grassmannian bundle. Let $\gamma_{n,k} \rightarrow \Gr^{\mathbb{C}}(\mathbf{\Delta}_{n,k},n)$ be the tautological bundle with fiber $\gamma_{n,k,L}=L$ over any point $L \in \Gr^{\mathbb{C}}(\mathbf{\Delta}_{n,k},n),$ viewed as a vector subspace of the corresponding fiber of $\mathbf{\Delta}_{n,k} \rightarrow \mathcal{Z}_n(Y).$  Since both $\gamma_{n,k}$ and $\pi^*_{n,k}(T_{\mathcal{Z}_n(Y)/Y})$ are vector bundles on $\Gr^{\mathbb{C}}(\mathbf{\Delta}_{n,k},n),$ we can form the vector bundle \newline $h:\Hom(\gamma_{n,k},\pi^*_{n,k}(T_{\mathcal{Z}_n(Y)/Y})) \rightarrow \Gr^{\mathbb{C}}(\mathbf{\Delta}_{n,k},n).$ The typical fiber of this bundle is $\Hom(\mathbb{C}^n,\mathbb{C}^{N_{n,k}-2k}),$ coming from the fact that for any $(w,S_w) \in \Gr^{\mathbb{C}}(\mathbf{\Delta}_{n,k},n),$ \[h^{-1}(w,S_w)=\Hom(\gamma_{n,k,S_w},\pi^*_{n,k}(T_{\mathcal{Z}_n(Y)/Y})),\] $\dim_{\mathbb{C}}(S_w)=n,$ and $\dim_{\mathbb{C}}(T_{\mathcal{Z}_n(Y)/Y,w})=N_{n,k}-2k.$

\begin{remark}\label{obv}
Let $V$ be a complex vector space and $W \subset V$ be a subspace. Let $p:V \to V/W$ be the linear projection map, $v \mapsto v \mod{W}.$ Define $\Gr^\mathrm{o}(V,r):=\{U \in \Gr^{\mathbb{C}}(V,r) \mid p|_U \mbox{ is injective}\}.$ Choose a subspace $T$ that is complementary to $W,$ i.e.\ $V=T \oplus W.$ Then,

\begin{equation*}
\begin{split}
\Gr^\mathrm{o}(V,r)&=\{U \in \Gr^{\mathbb{C}}(V,r) \mid p(U) \in \Gr^{\mathbb{C}}(V/W,r)\}\\
&=\{U \in \Gr^{\mathbb{C}}(V,r) \mid U \cap W=\{0\}\}\\
&=\{\Gamma(f) \mid f \in \Hom(S,W) \mbox{ for some }S \in \Gr^{\mathbb{C}}(T,r)\}.
\end{split} 
\end{equation*}

Since $T \simeq V/W,$ $\Gr^{\mathbb{C}}(T,r)$ and $\Gr^{\mathbb{C}}(V/W,r)$ are indistinguishable as complex manifolds. Denote by $\eta \rightarrow \Gr^{\mathbb{C}}(V/W,r)$ the usual tautological bundle. Then, the above gives a biholomorphism (and even a biregular algebraic map) from $\Gr^\mathrm{o}(V,r)$ to the total space of the vector bundle $\Hom(\eta,W) \rightarrow \Gr^{\mathbb{C}}(V/W,r),$ and here $W$ is to be understood as the trivial vector bundle $\Gr^{\mathbb{C}}(V/W,r) \times W \rightarrow \Gr^{\mathbb{C}}(V/W,r).$ Notice, however, that this identification depends on the choice of complementary subspace $T,$ and that the ``zero section" of the vector bundle moves with the choice of $T$ when it is mapped into $\Gr^\mathrm{o}(V,r).$

\begin{figure}
\centering
\begin{tikzpicture}
\draw[thick] (-3,-2) -- (-3,4) node[anchor=north west]{};
\draw[thick] (-3,4) -- (-2,3) node[anchor=north west]{};
\draw[thick] (-3,4) -- (-1,4) node[anchor=north east]{};
\draw[thick,->] (0,0) -- (3,0) node[anchor=north west] {$T$};
\draw[thick,->] (0,0) -- (-2.2,-2.04) node[anchor=north west] {$T$};
\draw[thick,->] (0,0) -- (2.5,-1.66) node[anchor=north west] {$S$};
\draw[thick,->] (0,0) -- (0,3) node[anchor=south east] {$W$};
\draw[thick,->] (0,0) -- (3,3) node[anchor=west] {$\Gamma(f)$};
\draw[dotted,thick] (1.28,-.8) -- (1.28,1.2);
\draw[dotted,thick] (0,2.2) -- (1.36,1.28);
\node[anchor=north] at (1.6,-.5) {$\footnotesize{x}$};
\node[anchor=east] at (0,2.2) {$\footnotesize{f(x)}$};
\node[anchor=west] at (1.28,1.28) {$\footnotesize{x+f(x)}$};
\node[anchor=west] at (-2.6,3.7) {$V$};
\end{tikzpicture}
\caption{A point $\Gamma(f)$ in $\Gr^\mathrm{o}(V,r),$ where $S \subset T$ is a \emph{moving} $r$-dimensional subspace, namely a point of the Grassmannian $\Gr^{\mathbb{C}}(T,r) \simeq \Gr^{\mathbb{C}}(V/W,r).$}
\end{figure}

\end{remark}

Given any vector bundle homomorphisms $F \in \Hom(\gamma_{n,k},\pi^*_{n,k}(T_{\mathcal{Z}_n(Y)/Y})),$ if $Id_{\gamma_{n,k}}$ is the identity morphism, we can produce another vector bundle morphism $Id_{\gamma_{n,k}} \oplus F: \gamma_{n,k} \to \gamma_{n,k} \oplus \pi^*_{n,k}(T_{\mathcal{Z}_n(Y)/Y}) \subset \pi^*_{n,k}(\mathbf{D}_{n,k}),$ and the graph of $F$ is then $\Gamma(F):=\im(Id_{\gamma_{n,k}} \oplus F).$ This can be regarded as a sub-bundle $\Gamma(F) \subset \gamma_{n,k} \oplus \pi^*_{n,k}(T_{\mathcal{Z}_n(Y)/Y}) \to \Gr^{\mathbb{C}}(\mathbf{\Delta}_{n,k},n)$ and the fiber over $(w,V)$ is the usual graph of the linear map $F(w,V)$ in $\Hom(V,T_{\mathcal{Z}_n(Y)/Y,w}),$ i.e.\ $\Gamma(F)_{(w,V)}=\im(Id_V + F(w,V))=\Gamma(F(w,V)).$ Let \[\Theta:\Hom(\gamma_{n,k},\pi^*_{n,k}(T_{\mathcal{Z}_n(Y)/Y})) \to \Lambda^2 \gamma^*_{n,k} \otimes \pi^*_{n,k}(T_{\mathcal{Z}_n(Y)}/\mathbf{D}_{n,k})\] be the bundle morphism $\Theta(F)=\theta|_{\Gamma(F) \times \Gamma(F)}.$ On fibers, this becomes a map 
\[\Theta(w,V):\Hom(\gamma_{n,k,V},T_{\mathcal{Z}_n(Y)/Y,w}) \to \Lambda^2 \gamma^*_{n,k,V} \otimes T_{\mathcal{Z}_n(Y),w}/\mathbf{D}_{n,k,w},\] where $\Theta(w,V)(f)=\theta(w)|_{\Gamma(f) \times \Gamma(f)}.$

Let $\rho: \Gro^\mathrm{o}_{n,k} \rightarrow \Gr^{\mathbb{C}}(\mathbf{\Delta}_{n,k},n)$ be the map with definition $\rho(w,V_w)=(w,d \pi(w)(V_w))$ for any $w \in \mathcal{Z}_n(Y)$ and $V_w \in \Gro^\mathrm{o}_{n,k,w}.$

\begin{lemma}\label{bigtheta}
If $w \in \mathcal{Z}_n(Y)$ and $V \in \Gr^{\mathbb{C}}(\mathbf{\Delta}_{n,k,w},n),$ then
\begin{enumerate}
 \item $\Theta(w,V)$ is linear, and
 \item $\mathbf{I}^\mathrm{o}_{n,k,w} \cap \rho^{-1}(w,V) \simeq \ker(\Theta(w,V)).$ 
\end{enumerate}
\end{lemma}

\begin{proof}
To prove 1., let $\zeta,\eta \in V.$ Remark \ref{the} implies that $\theta(w)(\zeta,\eta)=\theta(w)(f(\zeta),f(\eta))=0.$ From the bilinearity and anti-symmetry of $\theta(w),$ it then follows that $\theta(w)(\zeta + f(\zeta),\eta + f(\eta))=\theta(w)(\zeta,f(\eta))-\theta(w)(\eta,f(\zeta)).$ Apply this identity to $f+g$ to find that
\[\theta(w)(\zeta + (f+g)(\zeta),\eta + (f+g)(\eta))=\theta(w)(\zeta + f(\zeta),\eta + f(\eta)) + \theta(w)(\zeta + g(\zeta),\eta + g(\eta)).\] Clearly, for any $\lambda \in \mathbb{C},$ $\theta(w)(\zeta + \lambda f(\zeta),\eta + \lambda f(\eta))=\lambda \theta(w)(\zeta+f(\zeta),\eta+f(\eta)).$ The above also shows that $\Theta$ is a true vector bundle homomorphism.

The second claim will follow at once from the description of the fiber $\rho^{-1}(w,V)$ that we provide in the proof of the theorem below.
\end{proof} 

\begin{theorem*}\label{bigtheorem}
The spaces $\Gro^\mathrm{o}_{n,k}$ and $\mathbf{I}^\mathrm{o}_{n,k}$ have the structure of holomorphic affine linear bundles over the Grassmannian bundle $\Gr^{\mathbb{C}}(\mathbf{\Delta}_{n,k},n).$ 
\end{theorem*}

\begin{proof}
Let $(U_{\alpha})$ be a local trivialization of the holomorphic fiber bundle $\mathcal{Z}_n(Y) \xrightarrow{\pi_Y} Y$ so that $\mathcal{Z}_n(U_{\alpha})=\pi^{-1}_Y(U_{\alpha}) \simeq U_{\alpha} \times Q$ (cf.\ Remark \ref{tri}). The relative tangent bundle sequence need not be globally split. However, it is locally split with respect to a trivialization, meaning that, in particular, the short exact sequence \[0 \rightarrow T_{\mathcal{Z}_n(Y)/Y}\vert_{\mathcal{Z}_n(U_{\alpha})} \rightarrow T_{\mathcal{Z}_n(Y)}\vert_{\mathcal{Z}_n(U_{\alpha})} \rightarrow \pi^*_Y(T_Y)\vert_{\mathcal{Z}_n(U_{\alpha})} \rightarrow 0\] is split and induces the split sequence \[(*) \quad 0 \rightarrow T_{\mathcal{Z}_n(Y)/Y}\vert_{\mathcal{Z}_n(U_{\alpha})} \rightarrow \mathbf{D}_{n,k} \vert_{\mathcal{Z}_n(U_{\alpha})} \rightarrow \mathbf{\Delta}_{n,k}\vert_{\mathcal{Z}_n(U_{\alpha})} \rightarrow 0.\] Not all splittings of $(*)$ necessarily come from a trivialization of $\mathcal{Z}_n(Y) \xrightarrow{\pi_Y} Y.$

Define $G_{\alpha}:=\Gr^{\mathbb{C}}(\mathbf{\Delta}_{n,k},n)\vert_{\mathcal{Z}_n(U_{\alpha})}.$ Let $w_{\alpha} \in \mathcal{Z}_n(U_{\alpha})$ and $S_{w_{\alpha}} \in \Gr^{\mathbb{C}}(\mathbf{\Delta}_{n,k,w_{\alpha}},n).$ The direct sum of vector spaces $\gamma_{n,k}\vert_{G_{\alpha},S_{w_{\alpha}}} \oplus \pi^*_{n,k}(T_{\mathcal{Z}_n(Y)/Y}\vert_{\mathcal{Z}_n(U_{\alpha}),w_{\alpha}})$ is defined thanks to the splitting of $(*),$ and here $\gamma_{n,k}\vert_{G_{\alpha}}$ is the restriction of the tautological bundle to $G_{\alpha}.$ Notice that 

\begin{equation*}
\begin{split}
\rho^{-1}(w_{\alpha},S_{w_{\alpha}})&=\{V_{w_{\alpha}} \in \Gro^\mathrm{o}_{n,k,w_{\alpha}} \mid d \pi_Y(w_{\alpha})(V_{w_{\alpha}})=S_{w_{\alpha}}\} \\
&=\{V_{w_{\alpha}} \in \Gr^{\mathbb{C}}(\gamma_{n,k}\vert_{G_{\alpha},S_{w_{\alpha}}} \oplus \pi^*_{n,k}(T_{\mathcal{Z}_n(Y)/Y}\vert_{\mathcal{Z}_n(U_{\alpha}),w_{\alpha}}),n) \mid \\
& V_{w_{\alpha}} \cap \pi^*_{n,k}(T_{\mathcal{Z}_n(Y)/Y}\vert_{\mathcal{Z}_n(U_{\alpha}),w_{\alpha}}) = \{0\}\} \\
&=\{\Gamma(f) \mid f \in \Hom(\gamma_{n,k}\vert_{G_{\alpha},S_{w_{\alpha}}},\pi^*_{n,k}(T_{\mathcal{Z}_n(Y)/Y}\vert_{\mathcal{Z}_n(U_{\alpha}),w_{\alpha}}))\}.
\end{split}
\end{equation*}

Remark \ref{obv} is intended to provide some geometric intuition of the above identifications.

Consider the map \[\mathbf{t}:\Hom(\gamma_{n,k},\pi^*_{n,k}(T_{\mathcal{Z}_n(Y)/Y})) \times_{\Gr^{\mathbb{C}}(\mathbf{\Delta}_{n,k},n)} \Gro^\mathrm{o}_{n,k} \rightarrow \Gro^\mathrm{o}_{n,k},\] where the domain is the fiber product, and where the map $\mathbf{t}$ is defined on the fiber over $(w_{\alpha},S_{w_{\alpha}})$ by \[\Hom(\gamma_{n,k}\vert_{G_{\alpha},S_{w_{\alpha}}},\pi^*_{n,k}(T_{\mathcal{Z}_n(Y)/Y}\vert_{\mathcal{Z}_n(U_{\alpha}),w_{\alpha}}))\times \rho^{-1}(w_{\alpha},S_{w_{\alpha}}) \rightarrow \rho^{-1}(w_{\alpha},S_{w_{\alpha}}),\] \[(f,\Gamma(g)) \mapsto \Gamma(f+g).\] Now, if $\Gamma(f'), \Gamma(f'') \in \rho^{-1}(w_{\alpha},S_{w_{\alpha}}),$ $\Gamma(f'')=\Gamma(f+f')$ iff $f=f''-f',$ which is to say that the action of $\Hom(\gamma_{n,k}\vert_{G_{\alpha},S_{w_{\alpha}}},\pi^*_{n,k}(T_{\mathcal{Z}_n(Y)/Y}\vert_{\mathcal{Z}_n(U_{\alpha}),w_{\alpha}}))$ on $\rho^{-1}(w_{\alpha},S_{w_{\alpha}})$ is free and transitive. The map $\mathbf{t}$ thus realizes the fiber $\rho^{-1}(w_{\alpha},S_{w_{\alpha}})$ as an affine linear space modelled on the vector space \[\Hom(\gamma_{n,k}\vert_{G_{\alpha},S_{w_{\alpha}}},\pi^*_{n,k}(T_{\mathcal{Z}_n(Y)/Y}\vert_{\mathcal{Z}_n(U_{\alpha}),w_{\alpha}})).\]

There are a few subtleties. Since $\rho^{-1}(w_{\alpha},S_{w_{\alpha}})$ is not genuinely a vector space, an isomorphism with $\Hom(\gamma_{n,k}\vert_{G_{\alpha},S_{w_{\alpha}}},\pi^*_{n,k}(T_{\mathcal{Z}_n(Y)/Y}\vert_{\mathcal{Z}_n(U_{\alpha}),w_{\alpha}}))$ cannot be defined. The biholomorphisms between fibers of the bundles $\rho: \Gro^\mathrm{o}_{n,k} \rightarrow \Gr^{\mathbb{C}}(\mathbf{\Delta}_{n,k},n)$ and $h:\Hom(\gamma_{n,k},\pi^*_{n,k}(T_{\mathcal{Z}_n(Y)/Y})) \rightarrow \Gr^{\mathbb{C}}(\mathbf{\Delta}_{n,k},n)$ are not canonical. They depend on the local holomorphic splitting of the relative tangent bundle sequence. That $Y$ is a generic complex even dimensional manifold is the underlying reason for there being no natural splitting of the sequence, which in turn implies the non-naturality of the biholomorphisms.

Note that the sets $G_{\alpha},$ which are preimages of the open subsets $\mathcal{Z}_n(U_{\alpha}) \subset \mathcal{Z}_n(Y)$ under the continuous bundle projection $\pi_{n,k},$ form an open cover of the total space $\Gr^{\mathbb{C}}(\mathbf{\Delta}_{n,k},n).$ Let $\mathbf{\Gamma}:=\{\Gamma(f) \mid f \in \Hom(\mathbb{C}^n,\mathbb{C}^{N_{n,k}-2k})\},$ which is an affine linear space modelled after $\Hom(\mathbb{C}^n,\mathbb{C}^{N_{n,k}-2k}).$ For all $(w,S_w) \in G_{\alpha},$ we have a biholomorphism $\rho^{-1}(w,S_w)\simeq \mathbf{\Gamma},$ inducing a biholomorphism $\rho^{-1}(G_{\alpha}) \simeq G_{\alpha} \times \mathbf{\Gamma}.$ So the $G_{\alpha}$ are a trivialization of $\rho: \Gro^\mathrm{o}_{n,k} \rightarrow \Gr^{\mathbb{C}}(\mathbf{\Delta}_{n,k},n)$ as a holomorphic fiber bundle with typical fiber $\mathbf{\Gamma},$ and so $\rho: \Gro^\mathrm{o}_{n,k} \rightarrow \Gr^{\mathbb{C}}(\mathbf{\Delta}_{n,k},n)$ is a holomorphic affine linear bundle modelled on \newline $h:\Hom(\gamma_{n,k},\pi^*_{n,k}(T_{\mathcal{Z}_n(Y)/Y})) \rightarrow \Gr^{\mathbb{C}}(\mathbf{\Delta}_{n,k},n).$ 

Part 2 of Lemma \ref{bigtheta} follows from the biholomorphism $\rho^{-1}(w_{\alpha},S_{w_{\alpha}}) \simeq \Hom(\gamma_{n,k}\vert_{G_{\alpha},S_{w_{\alpha}}},\pi^*_{n,k}(T_{\mathcal{Z}_n(Y)/Y}\vert_{\mathcal{Z}_n(U_{\alpha}),w_{\alpha}})).$ So the fiber over $(w_{\alpha},S_{w_{\alpha}})$ of the sub-bundle $\rho|_{\mathbf{I}^\mathrm{o}_{n,k}}:\mathbf{I}^\mathrm{o}_{n,k} \rightarrow \Gr^{\mathbb{C}}(\mathbf{\Delta}_{n,k},n)$ is an affine linear space modelled on the vector subspace $\ker(\Theta(w_{\alpha},S_{w_{\alpha}}))$ of $\Hom(\gamma_{n,k,S_{w_{\alpha}}},\pi^*_{n,k}(T_{\mathcal{Z}_n(Y)/Y,w_{\alpha}})).$ The bundle $\mathbf{I}^\mathrm{o}_{n,k}$ is thus an affine bundle modelled on the (sub-)vector bundle $h|_{\ker{\Theta}}:\ker{\Theta} \to \Gr^{\mathbb{C}}(\mathbf{\Delta}_{n,k},n),$ where $\ker{\Theta}:=\{F \in \Hom(\gamma_{n,k},\pi^*_{n,k}(T_{Z_n(Y)/Y})) \mid \Theta(F)=0\}.$
\end{proof}

An immediate consequence of the theorem is that the total spaces $\Gro^\mathrm{o}_{n,k}$ and $\mathbf{I}^\mathrm{o}_{n,k}$ are of the same homotopy type. In particular, $\pi_i(\Gro^\mathrm{o}_{n,k},\mathbf{I}^\mathrm{o}_{n,k})=0,$ for all $i \geq 0.$ However, ordinary relative homotopy groups might not be capturing crucial information needed to study topological obstructions to integrability. Perhaps, a better perspective is furnished by fibered homotopy groups \cite{fiber}. In order to implement this viewpoint, it would be helpful to have a more complete understanding of the algebraic geometry of the fibers of $\Gro^\mathrm{o}_{n,k}$ and $\mathbf{I}^\mathrm{o}_{n,k},$ and of the topological constraints affecting the map $\im(\bar{\partial}_{J} F): X \rightarrow \mathbf{I}^\mathrm{o}_{n,k}.$ This will be the subject of future research. 

\subsection{The quotient bundle $\Gro^\mathrm{o}_{n,k}/\mathbf{I}^\mathrm{o}_{n,k}$}

Our next goal is to understand the quotient $\Gro^\mathrm{o}_{n,k}/\mathbf{I}^\mathrm{o}_{n,k},$ which is the vector bundle $\Hom(\gamma_{n,k},\pi^*_{n,k}(T_{\mathcal{Z}_n(Y)/Y}))/\ker{\Theta}.$ We begin with a characterization of its isomorphism type. 

\begin{proposition}\label{sphere}
There is an isomorphism of vector bundles on $\Gr^{\mathbb{C}}(\mathbf{\Delta}_{n,k},n),$
\[\Hom(\gamma_{n,k},\pi^*_{n,k}(T_{\mathcal{Z}_n(Y)/Y}))/\ker{\Theta} \simeq \Hom(\Lambda^2 \gamma_{n,k},\pi^*_{n,k}(T_{\mathcal{Z}_n(Y)}/\mathbf{D}_{n,k})).\]
\end{proposition}

\begin{proof}
The isomorphism 
\begin{equation*}
\begin{split}
T_{\mathcal{Z}_n(Y)/Y,w_y} & \simeq \Hom(S'_y,\Sigma'_y/S'_y) \oplus \Hom(S''_y,\Sigma''_y/S''_y) \oplus  \Hom(\Sigma'_y,\Sigma''_y) \oplus \Hom(\Sigma''_y,\Sigma'_y),
\end{split}
\end{equation*} 
depends on choices of direct sum decompositions $\Sigma'_y=Q_y' \oplus S'_y$ and $\Sigma''_y=Q_y'' \oplus S''_y.$ If we set \[A_{w_y}:=\Hom(S''_y,{\Sigma''_y}/{S''_y}) \oplus \Hom(\Sigma'_y,\Sigma''_y) \oplus \Hom(\Sigma''_y,S'_y),\] we get an isomorphism $T_{\mathcal{Z}_n(Y)/Y,w_y} \simeq A_{w_y} \oplus \Hom(\mathbf{\Delta}_{n,k,w_y},{\Sigma'_y}/{S'_y}).$ So for any $V \in \Gr^{\mathbb{C}}(\mathbf{\Delta}_{n,k,w_y},n),$ \[\Hom(V,T_{\mathcal{Z}_n(Y)/Y,w_y}) \simeq \Hom(V, \Hom(\mathbf{\Delta}_{n,k,w_y},{\Sigma'_y}/{S'_y})) \oplus \Hom(V,A_{w_y}).\] Now we show that the linear in $\Hom(V,T_{\mathcal{Z}_n(Y)/Y,w_y})$ function $\Theta(w_y,V)$ vanishes on $\Hom(V,A_{w_y}).$ This implies that \[\ker{(\Theta(w_y,V))}=\ker{\big(\Theta(w_y,V)}|_{\Hom(V, \Hom(\mathbf{\Delta}_{n,k,w_y},{\Sigma'_y}/{S'_y}))}\big) \oplus \Hom(V,A_{w_y})\] so that
\begin{equation*}
\begin{split}
\Hom(V,T_{\mathcal{Z}_n(Y)/Y,w_y})&/\ker{(\Theta(w_y,V))} \simeq  \\
& \Hom(V, \Hom(\mathbf{\Delta}_{n,k,w_y},{\Sigma'_y}/{S'_y}))/\ker{\big(\Theta(w_y,V)}|_{\Hom(V, \Hom(\mathbf{\Delta}_{n,k,w_y},{\Sigma'_y}/{S'_y}))}\big).
\end{split}
\end{equation*} Observe that for any $f \in \Hom(V,A_{w_y}),$ there exist unique $g \in \Hom(V,\Hom(S''_y,{\Sigma''_y}/{S''_y})),$ $g' \in \Hom(V,\Hom(\Sigma'_y,\Sigma''_y))$ and $g'' \in \Hom(V,\Hom(\Sigma''_y,S'_y))$ such that $f=g+g'+g''.$ If $(X^m)_{m=1}^n$ is a basis of $V,$ then in terms of the basis $\big(\frac{\partial}{\partial z_{ij}}\big)_{(i,j) \in (\mathcal{I})}$ of $T_{\mathcal{Z}_n(Y)/Y,w_y},$ 
\begin{equation*}
\begin{split}
f(X^m) &= g(X^m)+g'(X^m)+g''(X^m) \\
&=\sum_{k+1 \leq i \leq k+n, n+k+1 \leq j \leq 2k}  g(X^m)_{ij} \frac{\partial}{\partial z_{ij}} + \sum_{k+1 \leq i \leq 2k, 1 \leq j \leq k}  g'(X^m)_{ij} \frac{\partial}{\partial z_{ij}} \\ 
& +\sum_{n+1 \leq i \leq k, k+1 \leq j \leq 2k}  g''(X^m)_{ij} \frac{\partial}{\partial z_{ij}},
\end{split}
\end{equation*}
for any $1 \leq m \leq n.$ In particular, we see that $f(X^m)_{ij}=0$ for all $1 \leq i \leq n$ and $n+1 \leq j \leq 2k,$ from which it follows that \[\theta(w_y)(X^a+f(X^a),X^b+f(X^b))=-2\sum_{i=1}^n \sum_{j=n+1}^{2k} (X^a_j f(X^b)_{ij}-X^b_j f(X^a)_{ij})\frac{\partial}{\partial x_i}=0.\] 

Next we make use of the isomorphism $\Psi:\Hom(V, \Hom(\mathbf{\Delta}_{n,k,w_y},{\Sigma'_y}/{S'_y})) \to \Hom(V \otimes \mathbf{\Delta}_{n,k,w_y},{\Sigma'_y}/{S'_y}),$ $f \mapsto \Psi(f),$ where for any $a \in V$ and $b \in \mathbf{\Delta}_{n,k,w_y},$ $\Psi(f(a)(b))=f(a \otimes b),$ to re-express the quotient vector space of interest. If $\zeta,\eta \in V,$ then

\begin{equation*}
\begin{split}
\theta(w_y)(\zeta+f(\zeta),\eta+f(\eta))&=-2\sum_{i=1}^n \sum_{j=n+1}^{2k} (\zeta_j f(\eta)_{ij}-\eta_j f(\zeta)_{ij})\frac{\partial}{\partial x_i} \\
&=-2(f(\eta)\zeta-f(\zeta)\eta) \\
&=-2f(\eta \otimes \zeta - \zeta \otimes \eta).
\end{split}
\end{equation*}

Let $J=\langle a \otimes b - b \otimes a \mid a,b \in V \rangle \subset V \otimes V$ be the ideal such that $S^2(V)=V \otimes V/J.$ It is clear that $\Theta(w_y,V)(f)=0$ iff $\theta(w_y)(\zeta+f(\zeta),\eta+f(\eta))=0$ for all $\zeta,\eta \in V$ iff $f|_{J}=0.$ Observe that we have a vector space isomorphism \[\{f \in \Hom(V \otimes V, {\Sigma'_y}/{S'_y}) \mid f|_J=0\} \simeq \Hom(S^2(V),{\Sigma'_y}/{S'_y}).\] Chose a direct sum decomposition $V \oplus R \simeq  \mathbf{\Delta}_{n,k,w_y}$ so that

\[\Hom(V \otimes \mathbf{\Delta}_{n,k,w_y},{\Sigma'_y}/{S'_y}) \simeq \Hom(V \otimes V, {\Sigma'_y}/{S'_y}) \oplus \Hom(V \otimes R,{\Sigma'_y}/{S'_y}).\] Since $\Theta(w_y,V)$ is independent of $\Hom(V \otimes R,{\Sigma'_y}/{S'_y}),$ \[\ker{\big(\Theta(w_y,V)}|_{\Hom(V, \Hom(\mathbf{\Delta}_{n,k,w_y},{\Sigma'_y}/{S'_y}))}\big) \simeq \Hom(S^2(V),{\Sigma'_y}/{S'_y}) \oplus \Hom(V \otimes R,{\Sigma'_y}/{S'_y}).\] Finally, note that the decomposition $V \otimes V = S^2(V) \oplus \Lambda^2(V),$ implies that 

$\Hom(V \otimes V, {\Sigma'_y}/{S'_y}) \oplus \Hom(V \otimes R,{\Sigma'_y}/{S'_y})/\Hom(S^2(V),{\Sigma'_y}/{S'_y}) \oplus \Hom(V \otimes R,{\Sigma'_y}/{S'_y}) \simeq \Hom(\Lambda^2(V),{\Sigma'_y}/{S'_y}), \mbox{ where }T_{\mathcal{Z}_n(Y),w_y}/\mathbf{D}_{n,k,w_y} \simeq {\Sigma'_y}/{S'_y}.$

\end{proof}

The isomorphism also follows after noticing that $\Theta(w,V)$ is surjective. However, inasmuch as we can only see a less elegant proof of this fact, we do not include it here. 

\subsection{Concluding remarks}
Let $(X,J_X)$ be an $n$-dimensional compact almost complex manifold and $F:(X,J_X) \hookrightarrow \mathcal{Z}_n(Y)$ be a universal embedding. Consider the map $\tilde{F}:(X,J_X) \to \Gro^\mathrm{o}_{n,k},$ $\tilde{F}(x)=\im(\bar{\partial}_{J_X} F(x))$ (cf.\ section 2.2), and the following diagram
\[
\begin{tikzcd}
  (\rho \circ \tilde{F})^* \big(\Lambda^2 \gamma^*_{n,k} \otimes \pi^*_{n,k}(T_{\mathcal{Z}_n(Y)}/\mathbf{D}_{n,k}) \big) \arrow[r, rightarrow] \arrow[d]
    & \Lambda^2 \gamma^*_{n,k} \otimes \pi^*_{n,k}(T_{\mathcal{Z}_n(Y)}/\mathbf{D}_{n,k}) \arrow[dr] \\
  (X,J_X) \arrow[r, rightarrow, "\tilde{F}"] 
&\Gro^\mathrm{o}_{n,k} \arrow[r, rightarrow, "\rho"] \arrow[u, "\tilde{\Theta}"]
& \Gr^{\mathbb{C}}(\mathbf{\Delta}_{n,k},n),
 \end{tikzcd}\] where $\tilde{\Theta}$ is defined as $\tilde{\Theta}(\Gamma(F))=\theta|_{\Gamma(F) \times \Gamma(F)},$ and note that based on the discussion preceding Lemma \ref{bigtheta}, we have that $\Gro^\mathrm{o}_{n,k}=\{\Gamma(F) \mid F \in \Hom(\gamma_{n,k}, \pi^*_{n,k}(T_{\mathcal{Z}_n(Y)/Y}))\}.$ Indeed, by Proposition \ref{sphere}, $(\rho \circ \tilde{F})^* \big(\Lambda^2 \gamma^*_{n,k} \otimes \pi^*_{n,k}(T_{\mathcal{Z}_n(Y)}/\mathbf{D}_{n,k}) \big) \simeq (\rho \circ \tilde{F})^*(\Gro^\mathrm{o}_{n,k}/\mathbf{I}^\mathrm{o}_{n,k}).$ Now we can see that the Nijenhuis tensor $N_{J_X}$ of $J_X$ is essentially the pullback of the bundle morphism $\tilde{\Theta}$ by the lift $\tilde{F}$ of $F$ to $\Gro^\mathrm{o}_{n,k},$ i.e.\ 

\begin{proposition}
$N_{J_X}=4 \tilde{\Theta} \circ \tilde{F}.$
\end{proposition}

\begin{proof}
This is a direct application of Proposition 5.1 \cite{DG}, since for any $x \in X,$
\[(4 \tilde{\Theta} \circ \tilde{F}) (x)=4 \tilde{\Theta}(\im(\bar{\partial}_{J_X}F(x)))=4\theta(F(x))|_{\im(\bar{\partial}_{J_X}F(x)) \times \im(\bar{\partial}_{J_X}F(x))}=N_{J_X}(x).\]
\end{proof}

Our next goal will be to relate characteristic classes of the vector bundle $(\rho \circ \tilde{F})^*(\Gro^\mathrm{o}_{n,k}/\mathbf{I}^\mathrm{o}_{n,k})$ on $X$ to the integrability of $J_X.$ This goal is motivated by a reformulation of the strategy to study Yau's challenge--something that we will investigate in a future article.
\addcontentsline{toc}{section}{Appendix}
\section*{Appendix}
Here we explain why the above approach, using ordinary Chern classes, is unlikely to work for Yau's Challenge in the case of $S^6.$ As mentioned in Remark \ref{sphere6}, the real $6$-dimensional sphere can be embedded into the octonions. More precisely, $\mathbb{O}$ is generated by the unit elements $(e_j)_{j=0}^7,$ where $e_0=1,$ and $e_j$ is purely imaginary for all other $j.$ The imaginary part of $\mathbb{O}$ is $\im(\mathbb{O})=\Span_{\mathbb{R}}(e_j)_{j=1}^7.$ For any $u,v \in \mathbb{O},$ the octonion inner product of $u$ and $v$ is given by $\langle u,v \rangle:= Re(u \bar{v})=u \cdot v,$ where the right hand side is the Euclidean inner product of $u$ and $v,$ regarded as elements of $\mathbb{R}^8.$ So, $S^6$ can be conceived of as the set of all unit length imaginary octonions, $S^6=\{u \in \im(\mathbb{O}) \mid \langle u,u \rangle=1\}.$ For any $x \in \mathbb{O}$ such that $x^2=-1,$ the alternativity of $\mathbb{O}$ ensures that the endomorphism $J_{\mathbb{O}}(x):\mathbb{O} \to \mathbb{O},$ $J_{\mathbb{O}}(x)(\zeta)=\zeta x,$ turns $\mathbb{O}$ into a complex vector space. In this way, $S^6 \subset \mathbb{O}$ inherits an almost complex structure $J_{\mathbb{O}},$ precisely that which at $u \in S^6$ is given by $J_{\mathbb{O}}(u).$ The non-integrability of this structure is due to the non-associativity of the octonions. Indeed, for any $\zeta, \eta \in T_{\mathbb{O}}|_{S^6,u},$ the Nijenhuis tensor is 
\begin{equation*}
\begin{split}
N_{J_{\mathbb{O}}}(u)(\zeta,\eta)&=[\zeta,\eta]+[J_{\mathbb{O}}(u)(\zeta),J_{\mathbb{O}}(u)(\eta)]+J_{\mathbb{O}}(u)([\zeta,J_{\mathbb{O}}(u)(\eta)]+[J_{\mathbb{O}}(u)(\zeta),\eta])\\
&=\big(\zeta(\eta u)-(\zeta u)\eta\big)-\big(\eta(\zeta u)-(\eta \zeta)u\big) \quad \cite{Kir},
\end{split}
\end{equation*}

and moreover, $J_{\mathbb{O}}$ is homotopically unique \cite{Bryant}.

Recall the complex, quasi-projective manifold $\mathcal{Z}_3(\mathbb{C}^8)=\mathbb{C}^8 \times Q$ of complex dimension $N_{3,4}=46,$ where $Q=\{(S',S'',\Sigma',\Sigma'') \in F_{(1,4)}(\mathbb{C}^8) \times F_{(1,4)}(\mathbb{C}^8) \mid \Sigma' \oplus \Sigma'' = \mathbb{C}^8\},$ and where $F_{(1,4)}(\mathbb{C}^8)$ is the complex projective manifold of flags of signature $(1,4).$ For convenience, we denote the relative tangent bundle $T_{\mathcal{Z}_3(\mathbb{C}^8)/\mathbb{C}^8}$ by $T_Q$ so that $T_{\mathcal{Z}_3(\mathbb{C}^8)}=\pi^*(T_{\mathbb{C}^8}) \oplus T_Q,$ and the affine bundle projection $\rho:Gr^\mathrm{o}_{3,4} \to \Gr^{\mathbb{C}}(\Delta_{3,4},3)$ by $d\pi.$

Now, let $J$ be a hypothetical integrable almost complex structure on $S^6,$ and $J_t:[0,1] \times S^6 \to \End(T_{S^6})$ be a homotopy from $J_0=J_{\mathbb{O}}$ to $J_1=J.$ For each $t \in [0,1],$ we get a universal embedding $F_t:(S^6,J_t) \hookrightarrow \mathcal{Z}_3(\mathbb{C}^8)$ in the following way: if $J^{\mathbb{C}}_t \in \End(T_{\mathbb{O}}|^{\mathbb{C}}_{S^6})$ is the complexification of $J_t,$ $\Sigma'_t(u)=\Eig(J^{\mathbb{C}}_t(u),i),$ $\Sigma''_t(u)=\Eig(J^{\mathbb{C}}_t(u),-i),$ $S'_u(t)=\Eig(J^{\mathbb{C}}_t(u)|_{S_u},i),$ and $S''_u(t)=\Eig(J^{\mathbb{C}}_t(u)|_{S_u},-i),$ where $S_u=\{0\} \oplus N^{\mathbb{C}}_{S^6/\mathbb{O},u},$ then \[F_t(u)=(u,S'_u(t),S''_u(t),\Sigma'_t(u),\Sigma''_t(u)).\] So, $J_t$ defines the isotopy $F_t$ of universal embeddings, which in turn gives rise to a homotopy of uniquely defined for $J_t$ lifts $\tilde{F}_t:[0,1] \times (S^6,J_t) \to \Gr^\mathrm{o}_{3,4},$ where $\tilde{F}_t(u)=\im(\bar{\partial}_{J_t} F_t(u)).$ Let $G_t=d\pi \circ \tilde{F}_t.$ It turns out that 

\begin{proposition}\label{chernc}
All of the Chern classes of $G^*_t(\Gr^\mathrm{o}_{3,4}/\mathcal{I}^\mathrm{o}_{3,4})$ vanish.
\end{proposition}

Let us first observe that there is a vector bundle isomorphism

\begin{lemma}\label{l1}
$G^*_t(\Gr^\mathrm{o}_{3,4}/\mathcal{I}^\mathrm{o}_{3,4}) \simeq \Lambda^2 \overline{T^*_{S^6}} \otimes T_{S^6}.$
\end{lemma}

\begin{proof}
From Proposition \ref{sphere}, $\Gr^\mathrm{o}_{3,4}/\mathcal{I}^\mathrm{o}_{3,4} \simeq \Lambda^2 \gamma^*_{3,4} \otimes \pi^*_{3,4}(T_{\mathcal{Z}_3(\mathbb{C}^8)}/\mathcal{D}_{3,4}).$ By the transversality of the universal embedding $F_t:(S^6,J_t) \hookrightarrow \mathcal{Z}_3(\mathbb{C}^8),$ ${F_t}_*(T_{S^6}) \simeq T_{\mathcal{Z}_3(\mathbb{C}^8)}/\mathcal{D}_{3,4}.$ Then, since $F_t = \pi_{3,4} \circ G_t,$ $G^*_t\big(\pi^*_{3,4}(T_{\mathcal{Z}_3(\mathbb{C}^8)}/\mathcal{D}_{3,4})\big) \simeq T_{S^6}.$

Let $u \in S^6,$ $\zeta \in T_{S^6,u},$ and note that 

\begin{equation*}
\begin{split}
\bar{\partial}_{J_t} F_t(u)(\zeta)&=\frac{1}{2}\big(dF_t(u)(\zeta) + J_{\mathcal{Z}_3(\mathbb{C}^8)}(F_t(u)) \circ dF_t(u) \circ J_t(u)(\zeta)\big) \\
&=\frac{1}{2}\big(dF_t(u)(\zeta) + idF_t(u)(J_t(u) (\zeta))\big),
\end{split}
\end{equation*}

where $J_{\mathcal{Z}_3(\mathbb{C}^8)}$ is the given complex structure on $\mathcal{Z}_3(\mathbb{C}^8),$ which we may assume acts by multiplication by $i$ on fibers. The push-forward of $F_t$ at $u \in S^6$ is an $\mathbb{R}$-linear map \[dF_t(u)=Id_{T_{S^6,u}} \oplus \phi^t_u:T_{S^6,u} \to T_{\mathbb{C}^8,u} \oplus T_{Q,(S'_u(t),S''_u(t),\Sigma'_t(u),\Sigma''_t(u))},\] where $Id_{T_{S^6,u}}$ is the $\mathbb{C}^8$-component, and $\phi^t_u$ is the $Q$-component of the push-forward of $F_t$ at $u.$ Now, since the push-forward of $\pi:\mathcal{Z}_3(\mathbb{C}^8) \to \mathbb{C}^8$ is essentially the identity map, meaning that at any $w \in \mathcal{Z}_3(\mathbb{C}^8),$ $d\pi(w)=I_8 \oplus \mathbf{0}_{8 \times 38},$ we obtain that

\begin{equation*}
\begin{split}
d\pi(F_t(u))(\bar{\partial}_{J_t} F_t(u)(\zeta))&=\frac{1}{2}\big(d\pi(F_t(u))(\zeta+\phi^t_u(\zeta))+d\pi(F_t(u))(iJ_t(u)(\zeta)+i\phi^t_u(J_t(u)(\zeta)))\big) \\
&=\frac{1}{2}(\zeta+iJ_t(u)(\zeta)).
\end{split}
\end{equation*}

Therefore, $G_t(u)=\{\frac{1}{2}(\zeta+iJ_t(u)(\zeta)) \mid \zeta \in T_{S^6,u}\}=T^{0,1}_{S^6,u},$ $G^*_t \gamma_{3,4} = T^{0,1}_{S^6} \simeq \overline{T_{S^6}},$ and so $G^*_t \gamma^*_{3,4} \simeq \overline{T^*_{S^6}},$ where $\overline{T_{S^6}}$ is the conjugate bundle to $T_{S^6},$ i.e.\ $T_{S^6}$ endowed with the complex structure $-J_{\mathbb{O}}.$ And then,

\begin{equation*}
\begin{split}
G^*_t(\Gr^\mathrm{o}_{3,4}/\mathcal{I}^\mathrm{o}_{3,4})& \simeq \Lambda^2 G^*_t \gamma^*_{3,4} \otimes G^*_t\big(\pi^*_{3,4}(T_{\mathcal{Z}_3(\mathbb{C}^8)}/\mathcal{D}_{3,4})\big) \\
& \simeq \Lambda^2 \overline{T^*_{S^6}} \otimes T_{S^6}.
\end{split}
\end{equation*}
\end{proof}

The proof of Proposition \ref{chernc} is straightforward once we review some basic facts about Chern classes of complex vector bundles. Let $E \to X$ be a rank $m$ complex vector bundle with Chern roots $\alpha_i:=c_1(L_i),$ for $1 \leq i \leq m,$ and $V \to X$ be a rank $n$ complex vector bundle with Chern roots $\beta_j:=c_1(L'_j),$ for $1 \leq j \leq n.$ For any set of indeterminates $x_1,\dots,x_r,$ let $\sigma_k(x_1,\dots,x_r)$ be the $k$-th elementary symmetric polynomial in the $x_i.$ The Chern roots of $E\ \otimes V$ are $\alpha_i + \beta_j,$ where $1 \leq i \leq m$ and $1 \leq j \leq n.$ By definition, the $k$-th Chern class of $E \otimes V$ is 

\begin{equation*}
\begin{split}
c_k(E \otimes V)&=\sigma_k(\alpha_1 + \beta_1,\dots,\alpha_m + \alpha_n)\\
&=\sum_{1 \leq i_1 < \dots < i_k \leq m} \sum_{1 \leq j_1 < \dots < j_k \leq n} (\alpha_{i_1} + \beta_{j_1}) \dots (\alpha_{i_k} + \beta_{j_k}),
\end{split}
\end{equation*}

and we can write 

\[(\alpha_{i_1} + \beta_{j_1}) \dots (\alpha_{i_k} + \beta_{j_k})=\alpha_{i_1} \dots \alpha_{i_k} + \beta_{j_1} \dots \beta_{j_k} + p^{i_1,\dots,i_k}_{j_1,\dots,j_k},\] for some $p^{i_1,\dots,i_k}_{j_1,\dots,j_k} \in  \mathbb{Z}[\alpha_{i_1},\dots,\alpha_{i_k},\beta_{j_1},\dots,\beta_{j_k}].$ Then, since 

\begin{equation*}
\begin{split}
\sigma_k(\alpha_1 + \beta_1,\dots,\alpha_m + \alpha_n)&=\sum_{1 \leq j_1 < \dots < j_k \leq n} \sum_{1 \leq i_1 < \dots < i_k \leq m}  \alpha_{i_1} \dots \alpha_{i_k} + \sum_{1 \leq i_1 < \dots < i_k \leq m} \sum_{1 \leq j_1 < \dots < j_k \leq n} \beta_{j_1} \dots \beta_{j_k} + \\
& \sum_{1 \leq i_1 < \dots < i_k \leq m} \sum_{1 \leq j_1 < \dots < j_k \leq n} p^{i_1,\dots,i_k}_{j_1,\dots,j_k} \\
&=\sum_{1 \leq j \leq n} \sum_{1 \leq i_1 < \dots < i_k \leq m}  \alpha_{i_1} \dots \alpha_{i_k} + \sum_{1 \leq i \leq m} \sum_{1 \leq j_1 < \dots < j_k \leq n} \beta_{j_1} \dots \beta_{j_k} +\\
&\sum_{1 \leq i_1 < \dots < i_k \leq m} \sum_{1 \leq j_1 < \dots < j_k \leq n} p^{i_1,\dots,i_k}_{j_1,\dots,j_k} \\
&=n \sigma_k(\alpha_1,\dots,\alpha_m) + m \sigma_k(\beta_1,\dots,\beta_n) + P, 
\end{split}
\end{equation*}
where $P=\sum_{1 \leq i_1 < \dots < i_k \leq m} \sum_{1 \leq j_1 < \dots < j_k \leq n} p^{i_1,\dots,i_k}_{j_1,\dots,j_k} \in \mathbb{Z}[c_1(E),\dots,c_{k-1}(E),c_1(V),\dots,c_{k-1}(V)],$ it follows that

\begin{equation}
\begin{split}
c_k(E \otimes V)&=n c_k(E) + m c_k(V) + P.
\end{split}
\end{equation}

Moreover, if $m=3,$ 

\begin{equation}
c_3(\Lambda^2 E)=-c_3(E)+c_1(E)c_2(E).
\end{equation}

This is because $c_1(E)=\alpha_1+\alpha_2+\alpha_3,$ $c_2(E)=\alpha_1 \alpha_2 + \alpha_1 \alpha_3 + \alpha_2 \alpha_3,$ $c_3(E)=\alpha_1 \alpha_2 \alpha_3,$ and the Chern roots of $\Lambda^2 E$ are $\alpha_1 + \alpha_2, \alpha_1 + \alpha_3, \alpha_2 + \alpha_3,$ implying that

\begin{equation*}
\begin{split}
c_3(\Lambda^2 E)&=\sigma_3(\alpha_1 + \alpha_2, \alpha_1 + \alpha_3, \alpha_2 + \alpha_3) \\
&=(\alpha_1 + \alpha_2)(\alpha_1 + \alpha_3)(\alpha_2 + \alpha_3) \\
&=(\alpha_1 \alpha_3 \alpha_2 + \alpha_2 \alpha_1 \alpha_3) + (\alpha^2_1 \alpha_2 + \alpha^2_1 \alpha_3) + (\alpha_2 \alpha_1 \alpha_2 + \alpha_2 \alpha_3 \alpha_2) + (\alpha_1 \alpha^2_3 + \alpha_2 \alpha^2_3) \\
&=2c_3(E)+(\alpha^2_1 \alpha_2 + \alpha^2_1 \alpha_3 + \alpha_2 \alpha_1 \alpha_2 + \alpha_2 \alpha_3 \alpha_2 + \alpha_1 \alpha^2_3 + \alpha_2 \alpha^2_3)\\
&=2c_3(E)+(c_1(E)c_2(E)-3c_3(E))\\
&=-c_3(E)+c_1(E)c_2(E).
\end{split}
\end{equation*}

\noindent \emph{Proof of Proposition \ref{chernc}.} Since $H^{2k}(S^6;\mathbb{Z})=0$ for all $k \neq 3,$ the only not automatically trivial Chern classes of $G^*_t(\Gr^\mathrm{o}_{3,4}/\mathcal{I}^\mathrm{o}_{3,4})$ are the 3rd Chern classes. However, from Lemma \ref{l1} alongside with equations (1) and (2), we conclude that $c_3(G^*_t(\Gr^\mathrm{o}_{3,4}/\mathcal{I}^\mathrm{o}_{3,4}))=-3c_3(\overline{T^*_{S^6}})+3c_3(T_{S^6})=0. \quad\quad\quad\quad\quad\quad\quad\quad \qed$

\noindent
Gabriella Clemente                                

\noindent
e-mail: clemente6171@gmail.com

\begin{thebibliography}{9}
\bibitem{Bryant}
R.\ Bryant, \textit{S-S.\ Chern's study of almost complex structures on the six-sphere}, Asian J.\ Math.\ 10 (3) (2006), 561--605.

\bibitem{fiber}
M.\ Crabb and I.\ James, \textit{Fiberwise Homotopy Theory}, Springer Monogr.\ Math., Springer-Verlag, London, 1998.

\bibitem{directed} 
J-P.\ Demailly, \textit{Algebraic criteria for Kobayashi hyperbolic projective varieties and jet differentials.} Algebraic geometry: Santa Cruz 1995 (Santa Cruz, CA, 1995), 285--360, Proc.\ Sympos.\ Pure Math.\ 62, Part 2, Amer.\ Math.\ Soc., Providence, RI, 1997. 

\bibitem{CIME}
J-P.\ Demailly, \textit{Algebraic embeddings of complex and almost complex structures}, slides of a talk given at the CIME School on ``Non-K{\"a}hler geometry" (Cetraro, Italy, July 12, 2018), https://www-fourier.ujf-grenoble.fr/$\sim$demailly/manuscripts/cetraro$\_$2018-print.pdf.

\bibitem{DG} 
J-P.\ Demailly and H.\ Gaussier, \textit{Algebraic embeddings of smooth almost complex structures}, J.\ Eur.\ Math.\ Soc.\ 19 (2017), 3391--3419.

\bibitem{Kir}
L.O.\ Diaz, \textit{A note on Kirchoff's theorem for almost complex spheres, I}, arXiv:1804.05794.

\bibitem{YC}
S-T.\ Yau, \emph{Open problems in geometry.} Differential geometry: partial differential equations on manifolds (Los Angeles, CA, 1990), 1--28, Proc.\ Sympos.\ Pure Math.\ 54, Part 1, Amer.\ Math.\ Soc., Providence, RI, 1993. 
\end{thebibliography}
\end{document}